\documentclass[final]{siamltex}

\usepackage{amssymb}
\usepackage{amsmath,amsfonts,amssymb}
\usepackage{color}

\usepackage{graphicx}

\newtheorem{remark}[theorem]{Remark}

\usepackage{hyperref}

\hypersetup{
  colorlinks = false,
  urlcolor = black,
urlbordercolor={1 1 1},
}

\title{A Variational Approach for Continuous Supply Chain Networks\thanks{This work is partially supported by the NSF through grant EFRI-1024707, ``A theory of complex transportation network design".}}

\author{Ke Han \thanks{Department of Civil and Environmental Engineering, Imperial College London, London SW7 2BU, United Kingdom. \href{mailto:k.han@imperial.ac.uk}{\tt k.han@imperial.ac.uk} }
\and  Terry L. Friesz\thanks{Department of Industrial and Manufacturing Engineering, Pennsylvania State University, University Park 16802, USA. \href{mailto:tfriesz@psu.edu}{\tt tfriesz@psu.edu}}
\and Tao Yao\thanks{Department of Industrial and Manufacturing Engineering, Pennsylvania State University, University Park 16802, USA. \href{mailto:tyy1@engr.psu.edu}{\tt tyy1@engr.psu.edu}}}

\begin{document}
 \begin{center}
\textcolor{blue}{ARTICLE LINK:  http://epubs.siam.org/doi/abs/10.1137/120868943
\\  PLEASE CITE THIS ARTICLE AS\\ 
Han, K., Friesz, T.L., Yao, T., 2014. A variational approach for continuous supply chain networks. SIAM Journal on Control and Optimization 52 (1), 663-686.}
 \line(1,0){370}
 \end{center}

\maketitle

\begin{abstract}
We consider a continuous supply chain network consisting of  buffering queues and processors first articulated by \cite{Armbruster2} and analyzed subsequently by \cite{Armbruster1} and \cite{Armbruster4}. A model was proposed for such network by \cite{GHK1} using a system of coupling partial differential equations and ordinary differential equations. In this article, we propose an alternative approach based on a variational method to formulate the network dynamics. We also derive, based on the variational method, an algorithm that guarantees numerical stability, allows for rigorous error estimates, and facilitates efficient computations. A class of network flow optimization problems are formulated as mixed integer programs (MIP). The proposed numerical algorithm and the corresponding MIP are compared theoretically and numerically with existing ones  \cite{FGHKM, GHK1}, which demonstrates the modeling and computational advantages of the variational approach.
\end{abstract}



\pagestyle{myheadings}
\thispagestyle{plain}
\markboth{K. HAN, T. L. FRIESZ AND  T. YAO}{A VARIATIONAL APPROACH FOR CONTINUOUS  SUPPLY CHAIN NETWORK}

\section{Introduction}
\subsection{Modeling overview}
Manufacturing systems can be described by a number of mathematical models, which provide powerful tools to study and analyze the behavior of such systems under specific conditions. Among theses mathematical representations, we distinguish between static/stationary models and dynamic models; the latter has an inherent dependence on time and falls within the scope of this paper.

The dynamic models describe and predict time evolution of system states by introducing dynamics to different production steps. These models can be further categorized as {\it discrete event simulation} (DES) \cite{BCN}, and continuum models \cite{Armbruster2}. The DES  is an exact and computationally intensive method, in which the evolution of the system is viewed as a sequence of significant changes in time, called {\it events}, for each part (product) separately.  A cost-effective alternative to the discrete event models is fluid-based continuum network models represented by partial differential equations (PDEs). The continuum models seek to describe the system dynamics from an aggregate level and ignore local granularities. There has been a significant number of publications on the PDE formulation of traffic dynamics, for example, in \cite{CTM1, CTM2, GP, LW, Newell, Richards}. The continuum modeling technique was not applied to supply chain networks until recently, by the seminal work of \cite{Armbruster1, Armbruster2, Armbruster3, Armbruster4, CKW, DM, Kotsialos} and \cite{LaMarca}.

In particular, \cite{Armbruster2} derived, based on simple rules for releasing parts, a conservation law model for the density and flux (flow) of the parts in the production process: 
\begin{equation}\label{introclaw}
\partial_t\rho(t,\,x)+\partial_x\min\big\{V(x)\rho(t,\,x),\,\mu(x)\big\}~=~0\qquad t\in[0,\,T],\quad x\in[a,\,b]
\end{equation}
where the processor is expressed as a spatial interval $[a,\,b]$, and $\rho(t,\,x)$ denotes the spatial-temporal distribution of the product density. In other words, $\rho(t,\,x)$ measures the local product density (in number of products per unit distance) at time $t$ and location $x$.  $V(x)$ and $\mu(x)$ denote location-dependent processing speed and flow capacity, respectively.   The above PDE will be asymptotically valid in regimes where a substantial number of parts are present on the processor. It should be noted that the solution to the conservation law (\ref{introclaw}) can only be considered in the distributional sense, due to the discontinuous dependence of the flux function on $x$. This is easily seen from an example involving bottleneck: consider the flow capacity $\mu(a)$ at a point $a$, and assume $\mu(a^-)>\mu(a^+)$. If the flow $\mu(a^-)$ is saturated, a Dirac $\delta$-distribution will emerge in the density profile $\rho(t,\,\cdot)$ at location $x=a$, which corresponds to an active bottleneck. When this happens, an integral solution of \eqref{introclaw} does not exist. 

To overcome such theoretical difficulty, \cite{GHK1} proposed, in addition to the PDE formulation \eqref{introclaw}, a separate ordinary differential equation for the buffering queues immediately upstream of each processor, thus avoiding direct encounter of the $\delta$-distribution. This finally leads to a system of  conservation laws coupled with ordinary differential equations.  This supply chain model can be extended to incorporate general network topology \cite{HKP},  certain real-world production features such as multi-commodity or due-date production \cite{CSC}, and a class of network control and optimization problems \cite{GHK2}.

\subsection{Numerical techniques}

In  \cite{CSC, FGHKM, GHK1} and \cite{GHK2}, the numerical methods for solving the aforementioned coupling system of PDEs and ODEs is solved with an upwind finite difference scheme for the conservation laws and a forward Euler scheme for the ordinary differential equations. To guarantee stability of the explicit discretization scheme, the Courant-Friedrichs-Lewy condition \cite{LeVeque} must hold for the PDE:
\begin{equation}\label{CFL}
\Delta t~\leq~\min_{i} \Big\{{\Delta x_i\over V^i}\Big\} 
\end{equation}
where $\Delta t$ denotes the time step, index $i$ runs through every processor (link) of the supply chain network, $\Delta x_i$ and $V^i$ denote respectively the spatial step and the maximum processing speed on processor $i$. The solution method described above  encounters several numerical difficulties. First, the ODE representing the buffer queue has a discontinuous dependence on its unknown variable, this will be explained in more details later in Section \ref{CSCdef}. Certain modifications were proposed in the literature to remedy such numerical deficiency. For example, \cite{Armbruster4} proposed a smoothing parameter to revise the dynamics at queues. However, such modification could suffer from non-physical solutions to be illustrated in Section \ref{CSCcasestudy}. Second, the CFL condition for the PDE and the stiffness condition for the modified ODE from \cite{Armbruster4} imply a trade-off between numerical accuracy and computational efficiency, which could potentially increase the computational burden of the model.

In view of the above limitations, we propose in this article a reformulation of the same physical system  using cumulative production curves \cite{Wiendahl} and Hamilton-Jacobi equations. As we shall demonstrate subsequently, this alternative not only eliminates the abovementioned numerical issues, but also leads to an efficient `grid-free' algorithm and closed-form solution representation. Here `grid-free' means the PDE is solved without spatial discretization and without intermediate calculation of densities inside the processor.  The solution procedure of the Hamilton-Jacobi equation is a  variational method called the Lax  formula  \cite{BH, VT, Evans, LF, Lax}. The formula was originally proposed as a semi-analytic solution representation of the scalar conservation law and the Hamilton-Jacobi equation \cite{Evans, Lax}; its applications to fluid-based traffic modeling are recently investigated in \cite{BH, BH1, CC1, CC2, VT}. Using the variational approach, the viscosity solution of the Hamilton-Jacobi equation is formulated as an optimization problem which, depending on the specific form of the Hamiltonian, may be simplified or explicitly instantiated. The variational approach for the continuous supply chains is a powerful analytical and computational tool; and its advantages compared to the finite-difference schemes, to be fully established in the rest of this paper, are listed as follows.

\begin{enumerate}
\item The dynamics of the buffer queue and the processor can be simultaneously treated using a single Lax formula, thus avoiding separate modeling of these two components. 

\item The Lax formula yields a much lower numerical error in the solution than the finite-difference schemes.

\item The proposed algorithm is grid-free; in other words, there is no need to discretize the spatial domain of the PDE.

\item The proposed algorithm does not impose any constraints on the time step or uniformity of the time grid. Thus one has more flexibility in choosing the time grid for computational convenience. 

\item Our reformulation of the supply chain networks is free of the spatial variables originally appearing in the PDEs. This is arguably more general than the PDE-based system in terms of modeling assumptions and solution methods.
\end{enumerate}

This paper also presents a mixed integer program (MIP) formulation of the supply chain network optimization problem using the variational reformulation.  The mixed integer program is one of  the main formulations of supply chain optimization in the economic literature, see \cite{PW, VW}. Its connection with continuum network models brings new insights to the management of such networks; MIPs have been investigated quite extensively in the context of traffic flows and supply chains \cite{FHKM, FGHKM, MIPsignal, Z}. This article reveals a new relationship between MIPs and fluid-based continuum models from the point of view of variational principle.

It is natural to compare our proposed MIP with existing ones such as those proposed by \cite{FHKM, FGHKM}. The latter are based on a finite-difference discretization of the PDEs and ODEs. For the same reason mentioned before, the MIP based on the variational approach will allow more efficient computation and yields better solution quality. To reach the same level of numerical precision, the MIP we put forward requires much less (binary) variables than those based on finite-difference schemes; this will be established both theoretically and numerically  in this paper.

 The rest of the paper is organized as follows. Section \ref{CSCdef} briefly reviews the supply chain network model originally proposed by \cite{GHK1}, which consists of conservation laws and ODEs. In Section \ref{CSCvariational}, a variational approach is proposed to reformulate the system. The solutions are derived in closed form in both continuous and discrete time.  Section \ref{CSCnetwork} considers a network flow optimization problem based on the proposed variational approach and derives a mixed integer program. Such MIP is then compared to the MIP considered in \cite{FGHKM}. Finally, several numerical experiments are  presented in Section \ref{CSCnumerical} to illustrate the advantages of applying the variational approach.

\section{Supply chain network model}\label{CSCdef}
We begin with the articulation of the network model proposed by \cite{GHK1}. The supply chain model consists of separate modeling of the buffer queues (using ODEs) and processors (using PDEs). A precise description of the network is made via the notion of directed graph $G(\mathcal{A},\,\mathcal{V})$ with the set of edges (or arcs) $\mathcal{A}$ and the set of vertices (or nodes) $\mathcal{V}$.

\begin{definition}\label{networkdef} {\bf (Continuous supply chain network)}

1. A continuous supply chain network is represented as a directed graph $G(\mathcal{A},\,\mathcal{V})$ where each edge $e\in\mathcal{A}$ corresponds to an individual processor, each vertex $v\in\mathcal{V}$ represents the respective queue upstream of the processor. 

2.  Each processor $e\in\mathcal{A}$ is expressed as a spatial interval $[a^e,\,b^e]$, with $L^e=b^e-a^e$ being the length of the processor \footnote{Note that the spatial representation of the processor is somewhat imaginary and arbitrary. In an actual manufacturing process, the key quantity of interest is the processing time (or throughput time) of a part, which is assumed to be a constant in this paper. Having a virtual spatial dimension introduced to the dynamics enables us to invoke the PDE formulation.}. 

3.  Each processor possesses a queue, which is located at the vertex at the upstream end of the processor.

4 . The flow capacity $\mu^e$,  processing speed $V^e$ and throughput time $T^e=L^e/V^e$ of each processor $e\in\mathcal{A}$ are constants.

\end{definition}

For each $e\in\mathcal{A}$, let $\rho^e(t,\,x)$  denote the density of products at time $t\in[0,\,T]$ and location $x\in[a^e,\,b^e]$; let $q^e(t)$ denote the size of the queue upstream of this processor. Assume that products are fed to the buffer queue at the rate $\overline u^e(t)$, before they are released to the processor. The dynamics of the processor and the queue are governed by the following advection equation (\ref{eqn1}) and ODE (\ref{eqn2})-(\ref{feed}).
\begin{equation}\label{eqn1}
\partial_t\rho^e(t,\,x)+V^e\,\partial_x \rho^e(t,\,x)~=~0,\quad  \rho^e(0,\,x)~=~\rho^e_0(x),\quad (t,\,x)\in[0,\,T]\times[a^e,\,b^e]
\end{equation}
\begin{equation}\label{eqn2}
{d\over dt} q^e(t)~=~\overline u^e(t)-f^e\big(\rho^e(t,\,a^e)\big),\qquad q^e(0)~=~q^e_0
\end{equation}
\begin{equation}\label{feed}
f^e\big(\rho^e(t,\,a^e)\big)~=~\begin{cases}\min\{\overline u^e(t),\,\mu^e\}\qquad &q^e(t)~=~0\\
\mu^e\qquad &q^e(t)~>~0\end{cases}
\end{equation}
Equation (\ref{feed}) represents the rate at which products in the queue are released to the processor, i.e., the service rate. Notice that the PDE  (\ref{eqn1}) with the flow capacity constraint can be re-written as a scalar conservation law:
\begin{equation}\label{claw}
\partial_t\rho^e(t,\,x)+\partial_x\phi^e\big(\rho^e(t,\,x)\big)~=~0
\end{equation}
where the flux function satisfies 
\begin{equation}\label{phidef}
\phi^e(\rho)~=~\min\big\{V^e\rho,\,\mu^e\big\}
\end{equation}

A straightforward numerical scheme for solving system (\ref{eqn1})-(\ref{feed}) is to   apply space-time discretization and use the finite difference approximation, see \cite{Armbruster2, CSC, GHK1, GHK2}. For example, one could use an upwind scheme for conservation law (\ref{eqn1}) and forward Euler method for ODE (\ref{eqn2}). Moreover, to ensure numerical stability, one needs to impose the CFL-type constraints on the time step, see (\ref{CFL}).

One of the goals of this article is to provide an effective alternative to the numerical scheme mentioned above, namely, the variational method (Lax formula) \cite{BH, VT, Evans}. The Lax formula was proposed in \cite{Lax} for the study of scalar conservation laws of the form
\begin{equation}\label{quasi}
\partial_t u+ \partial_x f(u)~=~0
\end{equation}
of which the PDE (\ref{eqn1})  is just  a special case. The Lax formula expresses the weak entropy solution of (\ref{quasi}) as the solution of a minimization problem which, in our case,  can be expressed explicitly. Notice that  (\ref{quasi}) can be equivalently written as the following Hamilton-Jacobi equation
\begin{equation}\label{HJgeneral}
\partial_t U+ f\big(\partial_x U\big)~=~0,\qquad \hbox{where}\quad U(t,\,x)~=~\int_0^t u(s,\,x)\,ds
\end{equation}
We also wish to consider the following initial condition 
\begin{equation}\label{HJinitial}
U(0,\,x)~=~U_0(x)
\end{equation}
The classical Lax formula for the Cauchy problem (initial value problem) \eqref{HJgeneral}-\eqref{HJinitial} is stated below, the reader is referred to \cite{Evans} for a detailed derivation.
\begin{lemma}\label{laxlemma}{\bf (Lax formula)}
Assume that $f(\cdot)$ is convex, and that $U_0(\cdot)$ is Lipschitz continuous. Then the viscosity solution to the Cauchy problem \eqref{HJgeneral}-\eqref{HJinitial} is 
\begin{equation}\label{Laxoriginal}
U(t,\,x)~=~\inf_{y\in\mathbb{R}}\left\{t f^*\left({x-y\over t}\right)+U_0(y)\right\} \qquad t~>~0
\end{equation}
where $f^*(\cdot)$ is the Legendre transform of $f(\cdot)$: $f^*(v)=sup_u\{vu -f(u)\}$.
\end{lemma}

\section{Variational method}\label{CSCvariational}
In this section, we  apply the variational formulation to the system governed by  (\ref{eqn1})-(\ref{feed}), we will demonstrate the capability of the Lax formula to simultaneously handle dynamics of the  queue and processor, in a way consistent with (\ref{eqn1})-(\ref{feed}), and to reduce the complexity of the coupling PDEs and ODEs introduced in \cite{GHK1}. We first consider a single processor. For simplicity of notation, the superscript $e$  will be dropped for now. The following argument will be extended to a general supply chain network in Section \ref{CSCnetworkextension}. 

Denote the flux of products by $u(t,\,x)$, i.e., $u(t,\,x)\doteq \phi(\rho)$, where $\phi(\cdot)$ is defined in \eqref{phidef}. Recall the density-based scalar conservation law (\ref{claw}), which is rewritten as a PDE whose unknown variable is the flux:
\begin{equation}\label{clawu}
\partial_x u(t,\,x)+\partial_t g\big(u(t,\,x)\big)~=~0
\end{equation} 
where the function $u\mapsto g(u)=\rho$ is defined as the inverse of the function $\rho\mapsto\phi(\rho)= u$ on the interval $[0,\,\mu/V]$. To be precise, we have 
\begin{equation}\label{gdef}
g(u)~=~u/V\qquad u\in[0,\,\mu]
\end{equation}

\begin{remark}
Notice that such inversion of the unknown in the PDE is made possible only when the flux function $\phi(\cdot)$ is invertible. In addition, to show the equivalence between \eqref{clawu} and \eqref{claw}, one needs to invoke the definition of weak entropy solution (\cite{Bbook}). The reader is referred to  \cite{BH} for more detail on the inversion technique and \cite{FHNMY} for a formal proof of equivalence.
\end{remark}

Let us fix a time horizon $[0,\,T]$ for any given $T>0$. The processor of interest is represented by a spatial interval $[a,\,b]$, where $b-a=L$. We introduce the function $U(\cdot,\,\cdot): [0,\,T]\times[a,\,b]\mapsto \mathbb{R}_+$, defined as
$$
U(t,\,x)~\doteq~\int_{0}^t u(s,\,x)\,ds \qquad (t,\,x)\in[0,\,T]\times[a,\,b]
$$

The  conservation law (\ref{clawu}) can be equivalently written as a Hamilton-Jacobi equation with unknown $U(t,\,x)$: 
\begin{equation}\label{HJ}
\partial_x U(t,\,x)+ g\big(\partial_t U(t,\,x)\big)~=~0
\end{equation}

Define function $Q(\cdot): [0,\,T]\mapsto \mathbb{R}_+$ which measures the cumulative number of products that have arrived at queue by time $t$.  For what follows, $Q(\cdot)$ is naturally assumed to be non-decreasing and continuous except for countably many $t$. For reason that will become clearly later in Section \ref{secexplicitformula}, we extend $Q(\cdot)$ to the time interval $(-\infty,\,0)$, with value zero assigned. For the rest of this article, we consider the left-continuous version of $Q(\cdot)$, so that $Q(t)=Q(t-)$, $t\in\mathbb{R}$.  Now consider the Lipschitz continuous function:
\begin{equation}\label{sdef}
\overline U(t)~\doteq~\inf_{\tau\leq t}\big\{Q(\tau)+\mu(t-\tau)\big\}~\leq~Q(t)
\end{equation}
Clearly, $\overline U(t)$ measures the total number of products that have been released from the buffering queue to the processor by time $t$. Notice that $Q(t)-\overline U(t)$ measures the size of the queue upstream of the processor. We consider the following ``initial value" problem 
\begin{equation}\label{maineqn}
\begin{cases}
\partial_x U(t,\,x)+g\big(\partial_t U(t,\,x)\big)~=~0\\
U(t,\,0)~=~\overline U(t)
\end{cases}
\end{equation}
For $t>0$, the  viscosity solution to (\ref{maineqn}) is provided by the following variation of the Lax formula.
\begin{proposition}\label{Lax}
The solution to (\ref{maineqn}) is given by the following identities.
\begin{equation}\label{Laxformula}
U(t,\,x)~=~\inf_{\tau\in\mathbb{R}}\Big\{x\,g^*\Big({t-\tau\over x}\Big)+\overline U(\tau)\Big\}~=~\inf_{\tau\in\mathbb{R}}\Big\{x\,g^*\Big({t-\tau\over x}\Big)+Q(\tau)\Big\}
\end{equation}
where $g^*$ is the Legendre transform of $g$
\begin{equation}\label{legendre}
g^*(q)~=~\sup_{p\in[0,\,\mu]}\big\{q\,p-g(p)\big\}~=~
\begin{cases}
0\qquad  & q~\leq~{1\over V}
\\
\left(q-{1\over V}\right)\mu\qquad & q~>~{1\over V}
\end{cases}
\end{equation}
where $\mu$ denotes the flow capacity of the processor, $V$ denotes the processing speed.
\end{proposition}
\begin{proof}
Given the functional format of $g(\cdot)$ from \eqref{gdef}, the Legendre transform can be easily calculated as \eqref{legendre}. In order to apply the Lax formula given in Lemma \ref{laxlemma} to the problem \eqref{maineqn}, we readily notice, by switching the roles of $t$ and $x$, that problem \eqref{maineqn} is in fact an initial value problem. Adjusting the Lax formula to such Cauchy problem yields the first identity of \eqref{Laxformula}. To prove the second identity, we refer to reader to  \cite{BH}.
\end{proof}

Equation (\ref{Laxformula}) has several significant impacts.  (1) The first identity, which is an adjustment of the Lax formula to treat boundary value problems, provides a grid-free and semi-analytical solution representation of PDE \eqref{maineqn}. (2) The second identity suggests that the solution can be represented directly in terms of the inflow profile $Q(\cdot)$, without knowledge or intermediate computation of $\overline U(\cdot)$. As a result, the separating modeling of processor and buffer queue (\ref{eqn1})-(\ref{feed}) are no longer necessary, and the coupled PDEs and ODEs are replaced by the closed-form solutions provided by \eqref{Laxformula}. (3) The assumptions made on $Q(\cdot)$ suggest that the our methodological framework can accommodate very general inflow profile, which can even be a distribution.

\subsection{Explicit solution in continuous time}\label{secexplicitformula}

Given the simple and explicit functional forms of $g(\cdot)$ and $g^*(\cdot)$,  we are able to further simplify the semi-analytical expression \eqref{Laxformula} for the solution. The goal of this subsection is to derived closed-form solutions for the system in continuous time. A time-discretization will be introduced in the next subsection to express the solution in discrete time.

For reason that will become clear later, we focus our analyses on the family of {\it piecewise affine} (PWA) inflow profiles $Q(\cdot)$.  Such assumption is not too restrictive in application since any piecewise continuous function  can be approximated, to an arbitrary degree of accuracy, by piecewise affine functions. Notice that by (\ref{sdef}), $\overline U(\cdot)$ will also be piecewise affine. Given a fixed processor with length $L$, flow capacity $\mu$ and processing speed $V$, by setting $x=L$ formula (\ref{Laxformula}) becomes
\begin{equation}\label{eqn3}
U(t,\,L)~=~\min\left\{\displaystyle\inf_{\tau\geq t-L/V} Q(\tau),\qquad
\inf_{\tau<t-L/V}\big\{Q(\tau)-\mu\tau+(t-L/V)\mu\big\}\right\}
\end{equation}
We recall that $Q(\cdot)$ is non decreasing and left-continuous, therefore
\begin{align*}
\displaystyle\inf_{\tau<t-L/V}\big\{Q(\tau)-\mu\tau+(t-L/V)\mu\big\}~&\leq~\lim_{\tau\rightarrow (t-L/V)^-}\big\{Q(\tau)-\mu\tau+(t-L/V)\mu\big\}\\
~&=~Q(t-L/V)
~=~\inf_{\tau\geq t-L/V}Q(\tau)
\end{align*}
Now we can write (\ref{eqn3}) in a simplified form
\begin{equation}\label{eqn4}
U(t,\,L)~=~\inf_{\tau\leq t-L/V}\big\{Q(\tau)-\mu\tau\big\}+(t-L/V)\mu
\end{equation}
\begin{remark}
Identity (\ref{eqn4}) has a simple interpretation. Recall that the cumulative number of products that have been released from the queue  into the processor at time $t$ is given by (\ref{sdef}):
$$
\overline U(t)~\doteq~\inf_{\tau\leq t}\big\{Q(\tau)+\mu(t-\tau)\big\}
$$
Then for a product that exits the processor at time $t$, the time of its entry into the processor is $t-L/V$. In view of the natural first-in-first-out assumption, the total products that have exited the processor at time $t$ is equal to the total products that have been released into the processor at time $t-L/V$.  Replacing $t$ by $t-L/V$ in (\ref{sdef}) gives  rise to (\ref{eqn4}).
\end{remark}

The initial conditions for the buffer queue and the processor can be incorporate into our formula.  Let $q(t)$ be the queue size, and $\rho(t,\,x)$ be the density of products on the processor. Consider the following initial conditions
\begin{equation}\label{initialcond}
q(0)~=~q_0~\geq~0,\qquad \rho(0,\,\cdot\,)~=~\rho_0(\cdot)\in\mathcal{L}^1([a,\,b])
\end{equation}
where $\mathcal{L}^1([a,\,b])$ denotes the set of Lebesque integrable functions on the interval $[a,\,b]$.  Let $\Hat Q(\cdot):(-\infty,\,T)\mapsto \mathbb{R}_+$ be the left-continuous function defined by
\begin{equation}\label{Qhat}
\Hat Q(t)~=~\begin{cases}0\qquad &t~\leq~0\\
q_0+Q(t)\qquad & t~>~0\end{cases}
\end{equation}

We are now ready to state and prove one of the main results of this article, namely, the explicit instantiation of the variational principle (\ref{Laxformula}), with piecewise affine boundary profile $Q(\cdot)$ and initial conditions $q_0,\,\rho_0(\cdot)$.

\begin{proposition}\label{contform}{\bf (Continuous-time solution)} Given constants $\mu,\,L,\,V$ and piecewise affine inflow profile $Q(\cdot): [0,\,T]\mapsto \mathbb{R}_+$ with break points $\{\xi_i\in\mathbb{R}:\,i\in\mathcal{I}\}$, we associate with each time $t>L/V$ the finite set $\Omega_t\doteq\{\xi_i:~\xi_i\leq t-L/V\}\bigcup\{t-L/V\}$. Given the following initial conditions
\begin{equation}\label{initial}
q(0)~=~q_0,\qquad \rho(0,\,x)~=~\rho_0(x).
\end{equation}
we define $\Hat Q(\cdot)$ as in (\ref{Qhat}). Then the exit flow profile $U(t,\,L)$ can be written as
\begin{equation}\label{explicit}
U(t,\,L)~=~\begin{cases}
\displaystyle \int_{L-Vt}^L\rho_0(\zeta)\,d\zeta\qquad &t\in[0,\,L/V)\\
\displaystyle\min_{\tau\in\Omega_t}\big\{\Hat Q(\tau)-\mu\,\tau\big\}+(t-L/V)\mu+\int_0^L\rho_0(\zeta)\,d\zeta\qquad &t\in[L/V,\,T]
\end{cases}
\end{equation}
Moreover, $U(\cdot,\,L)$ is Lipschitz continuous.
\end{proposition}
\begin{proof}
 Notice that $L/V$ is the minimal time taken to traverse the processor, therefore at any positive time $t< L/V$, $U(t,\,L)$ is completely determined by the initial distribution $\rho_0(x)$ of products on the processor. A simple calculation using method of characteristics shows that for $0\leq t<L/V$,
\begin{equation}\label{pf1}
U(t,\,L)~=~\int_0^t V\rho(s,\,L)\,ds~=~\int_0^tV\rho_0(L-sV)\,ds~=~\int_{L-tV}^L\rho_0(\zeta)\,d\zeta
\end{equation}
For time beyond $L/V$, we argue that an initially positive queue $q_0$ can be treated as an upward jump in $Q(\cdot)$ at $t=0$, which is incorporated by replacing $Q(\cdot)$ with left-continuous, piecewise affine and non-decreasing function $\Hat Q(\cdot)$ defined in (\ref{Qhat}). Finally, in view of (\ref{eqn4}), it remains to show that 
\begin{equation}\label{pf2}
\inf_{\tau\leq t-L/V}\big\{\Hat Q(\tau)-\mu\tau\big\}+(t-L/V)\mu~=~\min_{\tau\in\Omega_t}\big\{\Hat Q(\tau)-\mu\tau\big\}+(t-L/V)\mu
\end{equation}
This is true since $\Hat Q$ is piecewise affine, the infimum in (\ref{pf2}) must be obtained at either one of the break points, or at $t-L/V$. See Figure \ref{discretefig} for an illustration. Since the Lax formula (\ref{Laxformula}) gives an viscosity solution to the Hamilton-Jacobi equation, $U(t,\,L)$ is Lipschitz continuous with Lipschitz constant $\mu$. 
\end{proof}

\begin{figure}[htbp]
\centering
\includegraphics[width=0.75\textwidth]{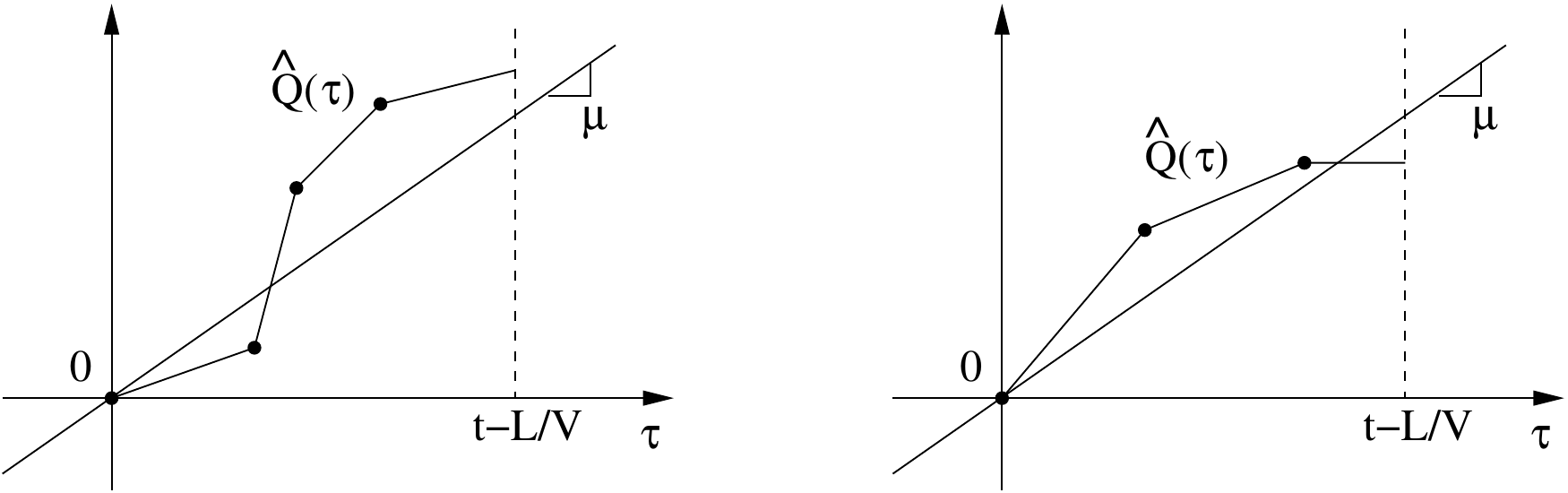}
\caption{Graphical representation of the formula (\ref{eqn4}). Since $\hat Q(\cdot)$ is piecewise affine, the infimum in the vertical difference between $\Hat Q(\tau)$ and $\mu\tau$ is obtained either at the break points of $\Hat Q(\cdot)$ (left), or at $t-L/V$ (right).}
\label{discretefig}
\end{figure}

The expression (\ref{explicit}) does not depend on any sort of approximation and is therefore exact. Notice that the feasible set $\Omega_t$ is finite, thus we have converted a continuous optimization problem into a discrete and finite form.   

\subsection{Explicit solution in discrete form}
The goal of this subsection is to derive the discrete version of (\ref{explicit}) for a given time grid.  This new algorithm will be applied to network simulation and optimization later in this paper.  We consider a time horizon $[0,\,T]$ for some $T>0$ and a uniform time grid $0=t_0<t_1<\ldots<t_N=T$ with step size $h$. 

Let $Q(\cdot)$ be any non-decreasing, left-continuous function. For notation convenience, let
$
Q_i\doteq Q(t_i-),~~\overline U_i \doteq\overline U(t_i),~~ 0\leq i \leq N
$.
Then $Q(\cdot)$ is approximated by the piecewise-affine function with break points $\{(t_i,\,Q_i)\}_{i=0}^N$. To derive an explicit  numerical scheme, we make one simplification by rounding $t_i-L/V$ to the nearest grid point to the left, and define integer constant $\Delta~\doteq~\lceil {L\over Vh}\rceil$.

We make note of the fact that the presence of the initial density $\rho_0(x)$ induces only integral terms to the Lax formula (\ref{explicit}), for which the error are quite standard in the  literature. Therefore, in the following statements of discrete formula and accompanying numerical error, we assume $\rho_0(x)\equiv 0$.

\begin{proposition}\label{discrete1}{\bf (Lax formula in discrete form)} Given constants $L,\,V,\,\mu$,  discrete values  $Q_i,\,i=0,\ldots, N$ and   initial data $q(0)=q_0>0,~\rho_0(x)\equiv 0$. Define $\Hat Q_0=0,~~\Hat Q_i=Q_i+q_0,~~1\leq i\leq N$. Let  $ U(t,\,L)$ satisfy (\ref{explicit}). Then the following discrete version of Lax formula 
\begin{equation}\label{explicitdisc}
U_{i,\,L}~=~\begin{cases}
0,\qquad & 0~\leq~i~<~\Delta;\\
\displaystyle\min_{0\leq j\leq i-\Delta}\big\{\Hat Q_j-\mu t_j\big\}+(t_i-L/V)\mu.\qquad &\Delta~\leq~i~\leq~N.
\end{cases}
\end{equation}
satisfies
\begin{equation}\label{nondecreasing}
0~\leq~U_{i,\,L}~\leq~U_{i+1,\,L}\qquad \Delta~\leq~i~\leq~N-1
\end{equation}
and
\begin{equation}\label{error}
0~\leq~U_{i,\,L}- U(t_i,\,L)~\leq~\Big(\lceil{L\over Vh}\rceil-{L\over Vh}\Big)\,h\mu, \qquad  \Delta~\leq~i~\leq~N
\end{equation}
In particular, if ${L\over Vh}$ is integer, (\ref{explicitdisc}) becomes exact.
\end{proposition}
\begin{proof}
We first verify   (\ref{nondecreasing}). Notice that
$
\displaystyle\min_{0\leq j\leq i-\Delta}\big\{\Hat Q_j-\mu t_j\big\}-\min_{0\leq j\leq i+1-\Delta}\big\{\Hat Q_j-\mu t_j\big\}
$
is equal to either $0$ or $\displaystyle\min_{0\leq j\leq i-\Delta}\big\{\Hat Q_j-\mu t_j\big\}-(\Hat Q_{i+1-\Delta}-\mu t_{i+1-\Delta})$, in the latter case
\begin{align*}
&\displaystyle\min_{0\leq j\leq i-\Delta}\big\{\Hat Q_j-\mu t_j\big\}-\min_{0\leq j\leq i+1-\Delta}\big\{\Hat Q_j-\mu t_j\big\}\\
~\leq~&\Hat Q_{i-\Delta}-\mu t_{i-\Delta}-(\Hat Q_{i+1-\Delta}-\mu t_{i+1-\Delta})~\leq~\mu\Delta
\end{align*}
thus $\displaystyle U_{i,\,L}-U_{i+1,\,L}=\min_{0\leq j\leq i-\Delta}\big\{\Hat Q_j-\mu t_j\big\}-\min_{0\leq j\leq i+1-\Delta}\big\{\Hat Q_j-\mu t_j\big\}-\Delta\mu\leq 0$. 
This shows monotonicity. To show non negativity, it suffices to check that $U_{\Delta,\,L}\geq 0$.

For error estimate (\ref{error}), recall $\Delta~\doteq~\lceil {L\over Vh}\rceil$, it follows from (\ref{explicit}) and the definition of $\Omega_t$ that
\begin{align*}
U(t_i,\,L)&~=~\min_{\tau\in\Omega_{t_i}}\big\{\Hat Q(\tau)-\mu\tau\big\}+(t_i-L/V)\mu\\
&~=~\displaystyle\min\left\{\min_{0\leq j\leq i-\Delta }\big\{\Hat Q(t_j)-\mu t_j\big\},~~ \Hat Q(t_i-L/V)-\mu(t_i-L/V)\right\}+(t_i-L/V)\mu\\
&~\leq ~U_{i,\,L}
\end{align*}
and $U(t_i,\,L)<U_{i,\,L}$ if and only if ~$\displaystyle
\Hat Q(t_i-L/V)-\mu(t_i-L/V)<\min_{0\leq j\leq i-\Delta}\big\{\Hat Q(t_j)-\mu t_j\big\}
$. In this case, we have the estimate
\begin{align*}
U_{i,\,L}-U(t_i,\,L)&~=~\min_{0\leq j\leq i-\Delta}\big\{\Hat Q(t_j)-\mu t_j\big\}-\Hat Q(t_i-L/V)+\mu(t_i-L/V)\\
&~\leq~\Hat Q(t_{i-\Delta})-\mu t_{i-\Delta}-\Hat Q(t_i-L/V)+\mu(t_i-L/V)\\
&~\leq~\mu(t_i-L/V- t_{i-\Delta})\\
&~=~\mu(t_i-L/V-t_i+\Delta h)~=~\mu h\big(\Delta-{L\over Vh}\big)
\end{align*}
This verifies (\ref{error}).
\end{proof}

\begin{remark}
Inequality  \eqref{error} implies the  uniform convergence of the numerical method as $h\rightarrow 0$. In fact, one may extend such error estimates to piecewise continuous $Q(\cdot)$, using standard results on linear interpolation. 
\end{remark}

Proposition \ref{discrete1} does not impose any constraint on the time grid in terms of step size and uniformity, except that $\Delta\geq 1$. This condition implies that the time step has to be less than or equal to the minimum processing time of the processor. According to \eqref{error}, in order to reduce the numerical error,  one may either reduce the step size $h$,  or less obviously, choose $h$ in a way such that the processing time $L/V$ is a multiple of $h$.

\subsection{Network model}\label{CSCnetworkextension}
We are in a position ready to extend the Lax formula to a general supply chain network. In view of  Definition \ref{networkdef}, we introduce a few additional notations. Given a node $v\in\mathcal{V}$, the set of incoming arcs and outgoing arcs are denoted by $\mathcal{I}^v$ and $\mathcal{O}^v$, respectively. In case $|\mathcal{O}^v|>1$, we call $v$ a {\it dispersive node}. We introduce the flow allocation rate $A^{v,\,e}(t)$ for each node $v$ and $e\in\mathcal{O}^v$:
\begin{equation}\label{allocationdef}
\begin{cases}
\displaystyle 0~\leq~A^{v,\,e}(t)~\leq~1,\qquad \sum_{e\in\mathcal{O}^v}A^{v,\,e}(t)~=~1 \qquad& \hbox{if}~|\mathcal{O}^v|>1\\
A^{v,\,e}(t)\equiv 1\qquad&\hbox{if}~|\mathcal{O}^v|=1
\end{cases}
\end{equation}

These allocation rates $A^{v,\,e}(t)$ describe the proportion of the flow coming from incoming links that is going to the outgoing link $e$. They will be later considered as control of the network product flow and are subject to optimization.  For a processor $e\in\mathcal{A}$, let $Q^e(t)$ be a non decreasing, left-continuous function describing the cumulative number of products arriving at the buffer queue of the processor. Let $W^e(t)$ be a non decreasing and Lipschitz continuous function describing the cumulative number of products that have left the processor. In addition, we fix the processor length $L^e$, the processing speed $V^e$ and the flow capacity $\mu^e$. Adapting Proposition \ref{contform} to network notations,  we obtain the following dynamics describing the supply chain network.

\begin{proposition} {\bf (Network model based on variational formulation)}  Given a supply chain network $G(\mathcal{A},\,\mathcal{V})$, consider, for every $e\in\mathcal{A}$, parameters $L^e$, $V^e$, $\mu^e$, and initial conditions $q^e(0)=q_0,\, \rho^e(0,\,x)=\rho^e_0(x)$.  Then we have the following variational formulations describing the dynamics of the system.
\begin{equation}\label{nl2}
{d \over d t} Q^e(t)~=~A^{v,\,e}(t)\sum_{\bar e\in\mathcal{I}^v}{d\over d t}W^{\bar e}(t)\quad \hbox{for almost every }t,\qquad \forall v\in\mathcal{V},~e\in\mathcal{O}^v
\end{equation}
\begin{equation}\label{nl3}
\Hat Q^e(t)~=~\begin{cases} 0,\qquad & t~\leq~0\\
q^e_0+Q^e(t)\qquad & 0~<~t~\leq~T
\end{cases}\qquad\qquad\forall e\in\mathcal{A}
\end{equation}
\begin{equation}\label{nl1}
W^e(t)=\begin{cases}
\displaystyle\int_{L^e-tV^e}^{L^e}\rho^e_0(\zeta)\,d\zeta \quad &t<{L^e\over V^e}\\

\displaystyle\inf_{\tau\leq t-{L^e\over V^e}}\big\{\Hat Q^e(\tau)-\mu^e\tau\big\}+\left(t-{L^e\over V^e}\right)\mu^e+\int_0^{L^e}\rho^e_0(\zeta)\,d\zeta \quad & t\geq {L^e\over V^e}
\end{cases}
\end{equation}

\end{proposition}

\begin{remark}
For any $e\in\mathcal{A}$, $W^e(t)$ is Lipschitz continuous by Proposition \ref{contform}, thus differentiable almost everywhere. Therefore the right hand side of (\ref{nl2}) is well defined for almost every $t$.  As a consequence, if  $e\in\mathcal{O}^v$ for some $v\in\mathcal{V}$, then $Q^e$ is also Lipschitz continuous. 
\end{remark}

The above system expresses the continuous-time solution of the network model proposed in \cite{GHK1}, where a system of coupling PDEs and ODEs were initially employed to describe the dynamics. As we shall see later, besides the closed-form solution, (\ref{nl1}) also yields better solution precision than the finite-difference schemes under minor condition.

\subsection{Some model extensions}
The conservation law \eqref{eqn1} used to represent the dynamic on a single processor is an advection equation and assumes a constant speed of all products regardless of their position or density. Such simple dynamics may not be quite realistic in certain applications, and it is the purpose of this section to introduce several model extensions and the corresponding variational methods.

\subsubsection{Location-dependent parameters}
A processor may comprise a sequence of individual stages, each one with a different processing speed and/or flow capacity. Such heterogeneity of parameters within a single processor may be captured by the following conservation law: 
\begin{equation}\label{mexteqn1}
\partial_t\rho(t,\,x)+\partial_x \min\left\{V(x)\rho(t,\,x),\,\mu(x)\right\}~=~0 \qquad (t,\,x)\in[0,\,T]\times[a,\,b]
\end{equation}
where unlike \eqref{eqn1}, the processing speed and capacity are dependent on the spatial parameter $x\in[a,\,b]$. A conservation law with an $x$-dependent flux function is often used by traffic analyst to model road heterogeneity such as lane change, curvature, road condition, etc., \cite{JinZhang}. We note that there exist variational methods for the conservation law (Hamilton-Jacobi equation) with $x$-dependent flux function (Hamiltonian); see, for example, \cite{Dolcetta}. In this paper we propose a simple treatment of \eqref{mexteqn1} by assuming  piecewise constant approximations of $V(x)$ and $\mu(x)$. By doing so, we decompose the dynamic on the processor into a set of homogeneous ones, each governed by an ODE for the queue and a PDE of the form \eqref{eqn1} with constant processing speed and capacity. Topologically, this means that we break each link in the network into smaller ones. The  variational approach proposed in this paper will apply to such new network.

\subsubsection{Density-dependent processing speed}
If the local processing speed is dependent on the density of products, the conservation law becomes 
\begin{equation}\label{mexteqn2}
\partial_t\rho(t,\,x)+\partial_x \left(V(\rho(t,\,x)) \rho(t,\,x)\right)~=~0\qquad (t,\,x)\in[0,\,T]\times[a,\,b]
\end{equation}
which is more in line with the fluid-like traffic dynamics \cite{LW, Richards}. In particular, it is assumed that the map $\phi(\cdot):\,\rho\mapsto V(\rho)\rho$ is concave and increasing. See Figure \ref{figCSCfd}. Notice that unlike traffic flow models, backward propagation of kinematic waves is not considered here; therefore there is no need to include the monotonically decreasing part of the {\it fundamental diagram} \cite{CTM1, CTM2, LW}. Similar techniques based on switching the roles of $t$ and $x$, as we demonstrate in this paper, will apply to conservation law \eqref{mexteqn2} and the corresponding Hamilton-Jacobi equation, and the variational solution representation can be consequently obtained. Due to space limitation, we omit technical details and refer the reader to \cite{BH} for a specific discussion. Notice that for general functional forms of $\phi(\cdot)$, the variational formulation amounts to a nonconvex optimization problem, for which closed-form solutions are hardly available. In a special case where $\phi(\cdot)$ is piecewise affine, it can be shown that the variational formulation becomes an optimization problem over a set of finite elements; this is directly related to the fact that there are only finite number of wave speeds, see the right part of Figure \ref{figCSCfd}.

\begin{figure}[h!]
\centering
\includegraphics[width=\textwidth]{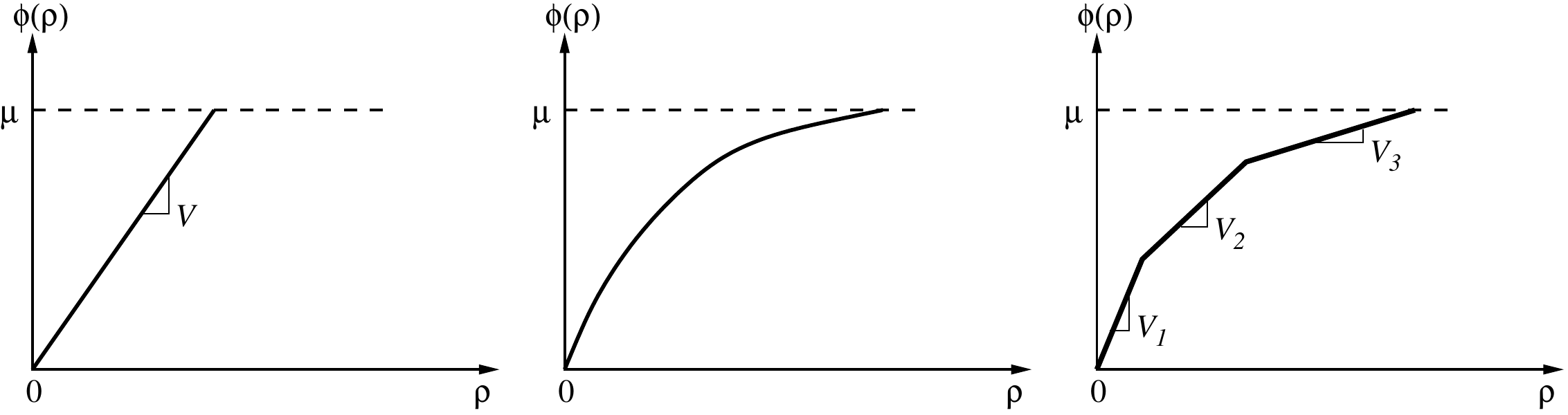}
\caption{Examples of concave and increasing flux function $\phi(\cdot)$ where $\rho$ denotes density, $\mu$ denotes flow capacity, and $V,\,V_1,\,V_2,\,V_3$ denote processing speeds. Left: the linear flux function employed in this paper. Middle: the smooth flux function which leads to a general continuous optimization problem. Right: the piecewise affine flux function which leads to an optimization problem over a finite set with cardinality equal to the number of affine pieces.}
\label{figCSCfd}
\end{figure}

\section{Optimizing network flow}\label{CSCnetwork}
This section defines and solves the following network optimization problem: given a general supply chain network as in Definition \ref{CSCdef}, with dynamics at processors and buffer queues described by the variational formulation \eqref{nl2}-\eqref{nl1}. Assume that at each dispersive node one can determine the allocation of flows by controlling $A^{v,\,e}(t)$, so that the network is optimized subject to some other constraints. A general objective functional of such optimization problem is
\begin{equation}\label{cost}
\min_{A^{v,\,e}(t)}\sum_{e\in\mathcal{A}}\mathcal{F}^e(Q^e,\,W^e)
\end{equation}
The continuous-time optimization problem is given by (\ref{cost}), subject to (\ref{nl2}), (\ref{nl3}) and (\ref{nl1}). It is shown in the next subsection that by properly time-discretize the system and by introducing binary variables, the optimization problem can be formulated as a {\it mixed integer program} (MIP).

\subsection{Time discretization}
We adapt the ``discretize-then-optimize" strategy by time-discretizing system (\ref{nl2})-(\ref{nl1}) and then optimizing the discrete system via mathematical programming techniques. To do this,  we prescribe a time grid $\{t_i\}_{i=0}^N$ with step size $h$, and adapt the following notations
$
Q^e_i\doteq Q^e(t_i)$,  $W^e_i\doteq W^e(t_i)$, $A^{v,\,e}_i\doteq A^{v,\,e}(t_i)$, $\Delta^e\doteq \lceil {L^e\over V^eh}\rceil
$.  

To avoid differentiation arising from (\ref{nl2}), we invoke the following lemma.
\begin{lemma}\label{lemmarate} Constraint (\ref{nl2}) can be equivalently expressed as
\begin{equation}\label{nl2'1}
\sum_{e\in\mathcal{O}^v}Q^e(t)~=~\sum_{\bar e\in\mathcal{I}^v}W^{\bar e}(t)
\end{equation}
\begin{equation}\label{nl2'2}
Q^e(t_1)~\leq~Q^e(t_2),\qquad \forall ~ t_1~<~t_2,~~ e\in\mathcal{O}^v
\end{equation}
\end{lemma}
\begin{proof}
``\eqref{nl2}$\Rightarrow$\eqref{nl2'1}, \eqref{nl2'2}". In view of  (\ref{allocationdef}), summing up (\ref{nl2}) over $e\in\mathcal{O}^v$ and integrating over time give rise to (\ref{nl2'1}). Moreover, 
$$
{d\over dt}Q^e(t)~=~A^{v,\,e}(t)\sum_{\bar e\in\mathcal{I}^v}{d\over dt}W^{\bar e}(t)~\geq~0\qquad\hbox{for almost every }t
$$
this establishes (\ref{nl2'2}). 

\noindent ``\eqref{nl2'1}, \eqref{nl2'2}$\Rightarrow$\eqref{nl2}". Differentiating (\ref{nl2'1}) gives
$$
\sum_{e\in\mathcal{O}^v}{d\over dt}Q^e(t)~=~\sum_{\bar e\in\mathcal{I}^v}{d\over dt}W^{\bar e}(t)
$$
define $A^{v,\,e}(t)~\doteq~{d\over dt}Q^e(t)/\sum_{\bar e\in\mathcal{I}}W^{\bar e}(t),\,e\in\mathcal{O}^v$, then it is straightforward to verify that (\ref{allocationdef}) holds.
\end{proof}

Note that (\ref{nl2'1}) and (\ref{nl2'2}) guarantee conservation of flux across the junction and non-negativity of the flux. After applying Lemma \ref{lemmarate}, information of the allocation rates $A^{v,\,e}$ is implicitly described by $Q^e$; and it is a matter of simple calculation to recover such information from the solution.

The discrete-time forms of (\ref{nl3}), (\ref{nl1}) are
\begin{equation}\label{optimize2}
\Hat Q^e_0~=~0,\qquad \Hat Q^e_i~=~q_0^e+Q_i^e,\qquad 1~\leq~i~\leq~N
\end{equation}
\begin{equation}\label{optimize3}
W^e_i~=~\begin{cases}0\qquad & i~<~\Delta^e\\
\displaystyle\min_{0\leq j\leq i-\Delta^e}\big\{\Hat Q^e_j-\mu^et_j\big\}+(t_i-L^e/V^e)\mu^e\qquad & i~\geq~\Delta^e
\end{cases}
\end{equation}

\noindent We introduce the variables $R^e_i$ for $e\in\mathcal{A},\,  \Delta^e\leq i\leq N$, defined recursively via the following.
\begin{equation}\label{dummydef}
R^e_{i}~=~\begin{cases}\Hat Q^e_{i-\Delta^e}-\mu^et_{i-\Delta^e}\qquad & i~=~\Delta^e\\
\min\Big\{R^e_{i-1},\quad \Hat Q^e_{i-\Delta^e}-\mu^et_{i-\Delta^e}\Big\}\qquad &\Delta^e<~i~\leq~N+\Delta^e
\end{cases} 
\end{equation}
Then the second equation of (\ref{optimize3}) becomes
\begin{equation}\label{optimize3'}
W_i^e~=~R_i^e+(t_i-L^e/V^e)\mu^e\qquad i~\geq~\Delta^e
\end{equation}
By virtue of this new variable $R_i^e$ and \eqref{dummydef}, for each time step $i$ we need to evaluate the ``min" function only once, thus reducing the number of variables and operations needed in the problem. In order to linearize the constraints involving the ``min" operator, we  introduce the binary  variables $\beta_i^e\in\{0,\,1\},\, i=\Delta^e+1,\ldots, \Delta^e+N$,  and equivalently write condition (\ref{dummydef}) as
\begin{equation}\label{binary}
R_{\Delta^e}^e~=~\Hat Q^e_0-\mu^e t_0,\quad
\begin{cases}
R_{i-1}^e+(\beta_i^e-1)M\leq R_i^e~\leq~R_{i-1}^e\\
\Hat Q_{i-\Delta^e}^e-\mu^et_{i-\Delta^e}-M\beta_i^e \leq R_{i}^e \leq \Hat Q_{i-\Delta^e}^e-\mu^et_{i-\Delta^e}
\end{cases}
\end{equation}
where $M$ is chosen to be a sufficiently large constant.  We are now ready to state the main result of this section.

\begin{theorem}\label{CSCtheorem}{\bf (MIP formulation of network optimization problem)}
Consider a supply chain network as in Definition \ref{networkdef}, with parameters $L^e$, $V^e$, $\mu^e$, and initial conditions $q^e(0)=q^e_0$, $\forall e\in\mathcal{A}$. Define a linear objective function $\mathcal{F}(Q^e,\,W^e,\,R^e)$. Given a time grid 
$0=t_0<t_1\ldots<t_N=T$ with step size $h$,
then the network optimization problem (\ref{cost}), (\ref{nl2}), (\ref{nl3}) and (\ref{nl1}) can be formulated as the following mixed integer program.
\begin{equation}\label{thm1}
\min_{Q^e_i,\,W^e_i,\,R^e_i}\mathcal{F}(Q^e_i,\,W^e_i,\,R^e_i)
\end{equation}
subject to
\begin{equation}\label{thm2}
\sum_{e\in\mathcal{O}^v}Q^e_i~=~\sum_{\bar e\in\mathcal{I}^v}W^{\bar e}_i,\qquad Q^e_{i-1}~\leq~Q^e_{i},\qquad 1\leq i\leq N
\end{equation}
\begin{equation}\label{thm3}
R_{\Delta^e}^e=-\mu^et_0,\quad R_{i-1}^e+(\beta_i^e-1)M\leq R_i^e \leq R_{i-1}^e,\quad \Delta^e+1\leq i\leq\Delta^e+N
\end{equation}
\begin{equation}\label{thm4}
 q_0^e+Q_{i-\Delta^e}^e-\mu^et_{i-\Delta^e}-M\beta_i^e\leq R_{i}^e \leq q_0^e+ Q_{i-\Delta^e}^e-\mu^et_{i-\Delta^e},~~ \Delta^e+1\leq i \leq \Delta^e+N
\end{equation}
\begin{equation}\label{thm5}
W_i^e~=~R_i^e+(t_i-L^e/V^e)\mu^e, \qquad  \Delta^e \leq i \leq N
\end{equation}
\begin{equation}\label{thm6}
\Delta^e~\doteq~\lceil{L^e\over V^eh}\rceil,\qquad \beta_i^e\in\{0,\,1\},\qquad e\in\mathcal{A},\qquad v\in\mathcal{V}
\end{equation}
where $M$ is a sufficiently large constant.
\end{theorem}
\begin{proof}
We have already established (\ref{thm2}) from Lemma \ref{lemmarate} and (\ref{thm5}) from (\ref{optimize3'}). (\ref{thm3}) and (\ref{thm4}) are immediate consequences of (\ref{optimize2}) and (\ref{binary}).
\end{proof}

\vspace{0.15 in}

\begin{remark}\label{remark1}
If $|\mathcal{O}^v|~=~1$, i.e. node $v$ is non-dispersive, the non-negative flow condition $Q_{i-1}^e\leq Q_i^e$ in (\ref{thm2}) is automatically satisfied since $W_i^{\Bar e},~\Bar e\in\mathcal{I}^v$ are guaranteed to be non-decreasing with respect to the time step $i$, due to Proposition \ref{discrete1}. Therefore the non-negative flow constraints can be dropped for such node.
\end{remark}

\vspace{0.15 in}

\begin{remark}
The objective function \eqref{thm1} is, in its own form, an arbitrary real-valued function of $Q_i^e$, $W_i^e$ and $R_i^e$, for all $0\leq i\leq N$, $e\in\mathcal{A}$. Throughout this paper, the objective function is chosen to be linear for the ease of computation, see Section \ref{nofinitebuffer} and \ref{sectime}. However, in a manufacturing environment with different cost scenarios, the objective function is often nonlinear, nonconvex and even nonsmooth. If the objective function is smooth and convex, then the resulting program can still be solved efficiently with commercial software. If the objective function is nonsmooth but is convex and admits well-defined subgradients, then using branch and bound, the relaxed problem can still be handled relatively well by ellipsoidal method, analytic center cutting-plane method, etc. 
\end{remark}

\subsection{Model extensions}\label{CSCmodelextension}
 We provide some discussion on two straightforward extensions of the proposed MIP formulation.
 
 \subsubsection{Finite buffers}
In a realistic manufacturing process, the capacity of the buffering queue is usually limited.  Articulation of this type of condition in the optimization procedure requires the queue size $q^e(t)$ to be expressed explicitly.  Recall that the queue size can be written as $q^e(t)=Q^e(t)-\overline U^e(t)$, where $\overline U^e(t)$ is given by (\ref{sdef}).  Thus we have
\begin{equation}\label{queuelength}
q^e(t)~=~Q^e(t)-\overline U^e(t)~=~(Q^e(t)-\mu^e t)-\inf_{\tau\leq t}\big\{Q^e(\tau)-\mu^e\tau\big\}
\end{equation}
The discrete-time expression for the queue reads 
\begin{equation}\label{queuedisc}
q^e_i~=~Q^e_i-\mu^e t_i-\min_{0\leq j\leq i}\big\{Q^e_j-\mu^e t_j\big\}~=~Q^e_i-\mu^e t_i-R^e_{i+\Delta^e},\qquad i=0,\,\ldots,\,N
\end{equation}
The finite buffer constraint can then be implemented in the mixed integer program by adding  the following linear constraints
\begin{equation}\label{thm8}
Q_i^e-\mu^e t_i-R_{i+\Delta^e}^e~\leq~C_q^e,\qquad e\in\mathcal{A},\quad 0\leq i\leq N
\end{equation}
where $C_q^e$ denotes the buffer queue capacity on link $e$.

\subsubsection{Inventory cost}\label{subsecinventory} In certain cases the handling of products in the buffer queue may incur additional costs; these may include fixed costs (e.g. warehouse rents) and variable costs (e.g. product degradation). Such situation can be easily handled in our framework by adding the following cost function to the objective function:
\begin{equation}\label{extendcost}
\sum_{e\in\mathcal{A}}\alpha^e_q\sum_{i=0}^N(Q^e_i-\mu^e t_i-R_{i+\Delta^e}^e),\qquad \alpha_q^e\geq 0,\quad e\in\mathcal{A}
\end{equation}  
where constant $\alpha_q^e$ measures the cost per unit storage. With the revised objective function, the decision variables may further include the network inflow profiles which allow the buffer queues to remain minimum, while maintaining a maximum throughput of the supply chain network.

\subsection{Comparison with existing approaches}\label{CSCcomparison}

The MIP (\ref{thm1})-(\ref{thm6}) differs from the one proposed by \cite{FGHKM} in a number of ways. First of all, the governing equations for the system in our paper are derived from a variational perspective, and the main variables are cumulative product counts, which are related to each other via the Hamilton-Jacobi equations and the junction conditions. In contrast to the model of \cite{GHK1, GHK2}, the dynamics of the processor and the buffering queue are simultaneous handled by the Lax formula. Our approach accepts discontinuous boundary datum $Q(\cdot)$, which implies the presence of a delta-distribution in the flow; such situation cannot be handled by the conservation law without a separate modeling of the queue. This intuitively explains why in the variational formulation, an explicit treatment of the queue is unnecessary.  The Lax formula not only avoids the numerical issue arising from the discontinuity in the ODE \eqref{eqn2}-\eqref{feed} (a detailed study of such issue will be provided in Section  \ref{CSCcasestudy}), but also facilitates efficient simulation and optimization of the network by reducing the number of variables needed and memory usage, due to the `grid-free' nature of the algorithm. The conservation law-based models, on the other hand, rely on a two-dimensional grid and are restricted by the CFL conditions, thus could potentially lead to large systems which are computationally expensive.

Next we compare the number of variables used in the MIP (\ref{thm1})-(\ref{thm6}) with those presented in \cite{FGHKM} with a two-point spatial discretization. Assume the numbers of arcs in the network is $|\mathcal{A}|$, and that the number of time intervals is $N$. Our proposed MIP has $3N|\mathcal{A}|$ real variables and $N|\mathcal{A}|$ binary variables; the MIP of \cite{FGHKM} has the same number of real and binary variables. However, in the latter approach, a two-point spatial discretization may be too coarse to properly represent the PDE. If the spatial grid is to be refined by a factor of $n$, the CFL condition \eqref{CFL} implies that the number of variables needed for the spatiotemporal grid will increase by a factor of $n^2$, and the number of binary variables will also increase by a factor of $n$. Such fact reveals a trade-off between numerical accuracy and computational efficiency for the conservation law models and the MIP built upon them. The variational approach, on the other hand, does not invoke spatial discretization and has no restrictions on the time step. Therefore, to achieve the same level of numerical precision, our proposed MIP is arguably much smaller in size compared to that of \cite{FGHKM}, and can be computed more efficiently.

\section{Numerical example}\label{CSCnumerical}
In this section, we conduct a sequence of numerical studies of the variational approach and the resulting MIP, with modeling extensions discussed in Section \ref{CSCmodelextension}. Consider a test network shown in Figure \ref{seven}, which consists of seven arcs (processors)  $a-g$, and four nodes $1-4$. The primary control variables are the time-varying flow allocation rate $A^{1,\,b}(t)$ and $A^{2,\,e}(t)$ at dispersive nodes $1$ and $2$, respectively. The time horizon is fixed to be $[0,\,10]$. Network parameters employed in our numerical test are shown in Table \ref{tabparameters}. In addition, we assume zero initial conditions: $q_0^e=0,\, \rho_0^e(x)=0,\,\,\forall e\in\mathcal{A}$.

\begin{figure}[htbp]
\centering
\includegraphics[width=0.48\textwidth]{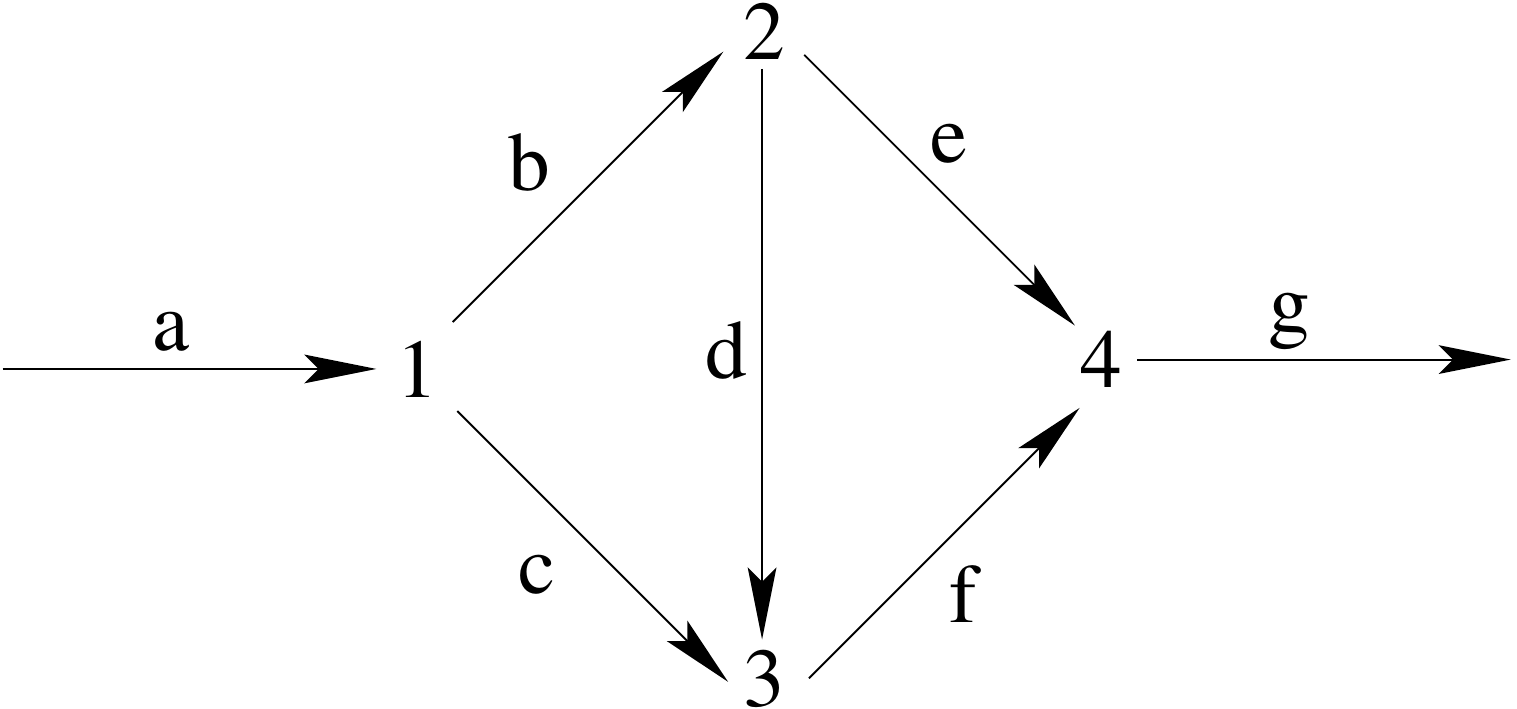}
\caption{A test network consisting of seven arcs and four nodes}
\label{seven}
\end{figure}

\begin{table}[h!]\label{tablepara}
\centering
\begin{tabular}{|c|c|c|c|c|c|c|c|}
\hline
Processor  & $a$ & $b$  & $c$  &  $d$  &  $e$  &  $f$  &  $g$ \\
\hline\hline
$L^e$         &   2    &    2    &   2    &     2    &    2    &    2    &    2    \\
\hline
$V^e$         &   2    &    1    &   2    &     4    &    2    &    2    &    2    \\
\hline
$\mu^e$     &  15  &   6    &  5    &    4    &   3.5  &   8    &  14   \\
\hline
\end{tabular}
\caption{Processor parameters}
\label{tabparameters}
\end{table}

All mixed integer programs presented in this section were solved with ILOG Cplex 12.1.0 \cite{cplex}, on the HPC platform with Intel Xeon X$5675$ Six-Core 3.06 GHz processor, provided by the Penn State High Performance Computing Systems \cite{HPC}.

\subsection{Without buffer size constraints}\label{nofinitebuffer}
As our first scenario, we assume that the flow $f^a(t)$ into the network through processor $a$ is fixed and given by 
$$
f^a(t)~=~
\begin{cases}
37.5 \qquad & 0\leq t\leq 2
\\
0\qquad &\hbox{otherwise}
\end{cases}
$$
It is assumed that all the buffers have infinite capacity. The objective is to maximize the total throughput of the network within the time horizon, namely 
\begin{equation}\label{obj1}
W^g(10)
\end{equation}
 where the notations have the same meaning as before. The optimal allocation rates $A^{1,\,b}(t),\,A^{2,\,e}(t)$ are computed from the MIP formulation introduced by Theorem \ref{CSCtheorem}, and are illustrated graphically in Figure \ref{figinfA1} and Figure \ref{figinfA2}.  In the optimized network flow profile, only processors $a$, $b$ and $c$ have non-empty buffering queues associated with them; this is shown in Figure \ref{figinfqueue}.  The optimal value of the throughput $W^g(10)$ is 58.75.

\begin{figure}[htbp]
\begin{minipage}[b]{.49\textwidth}
\centering
\includegraphics[width=1\textwidth]{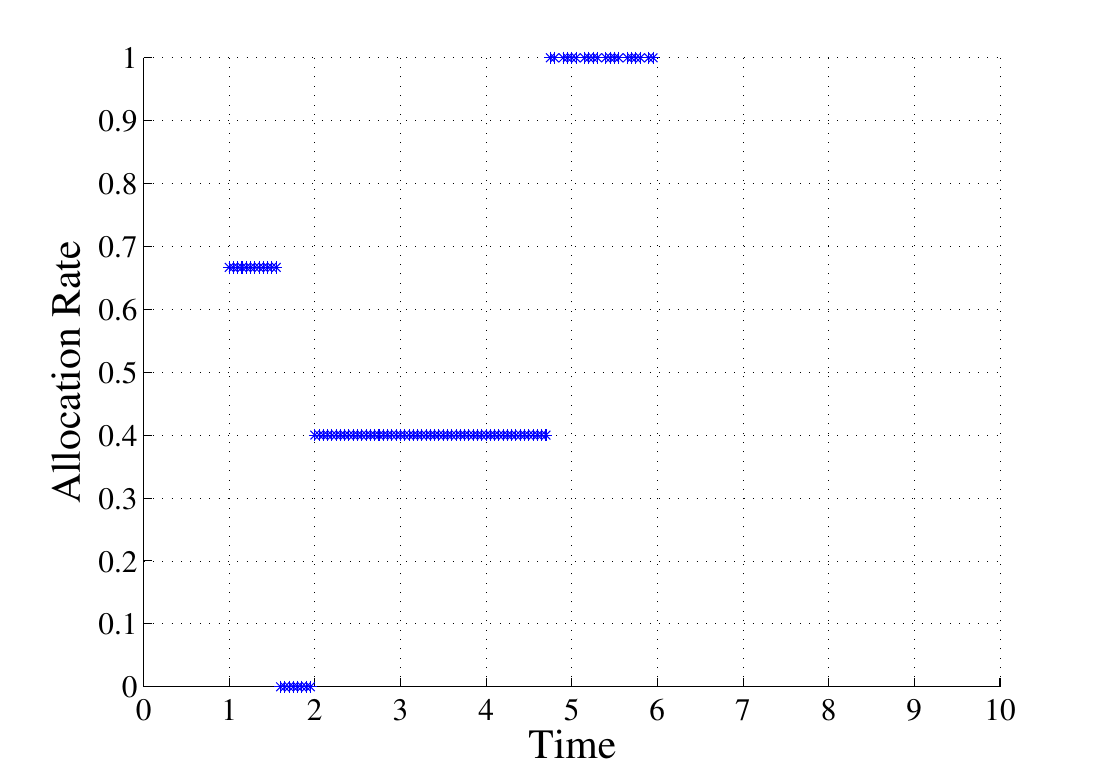}
\caption{\small Without buffer size constraints: optimal allocation rate $A^{1,\,b}$.}
\label{figinfA1}
\end{minipage}
\hspace{0.001cm}
\begin{minipage}[b]{.49\textwidth}
\centering
\includegraphics[width=1\textwidth]{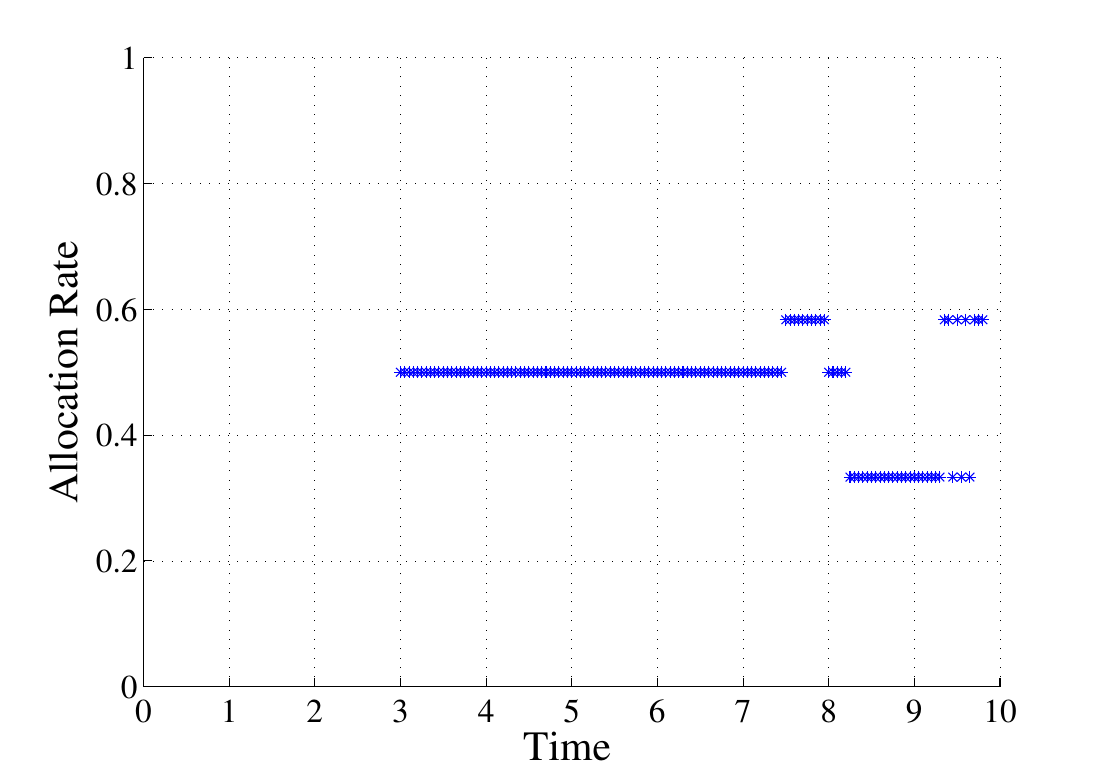}
\caption{\small Without buffer size constraints: optimal allocation rate $A^{2,\,e}$.}
\label{figinfA2}
\end{minipage}
\end{figure}

\begin{figure}[h!]
\centering
\includegraphics[width=0.75\textwidth]{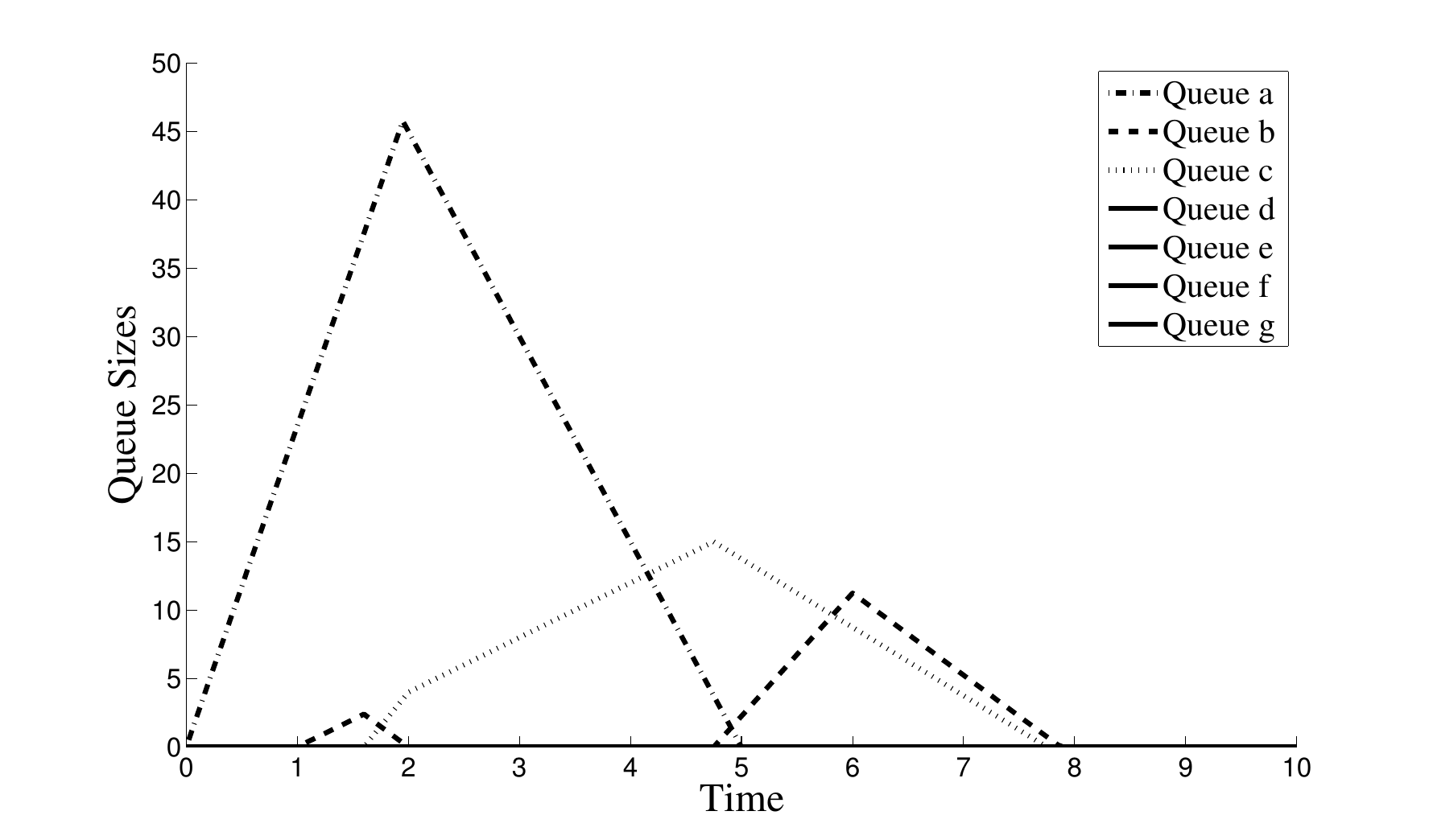}
\caption{Without buffer size constraints: queues upstream of each processor}
\label{figinfqueue}
\end{figure}

 \subsection{With finite buffer constraints}\label{finitebuffer}
We consider additional buffer capacity constraints in the MIP, namely, $q^b_i\leq 10$,~ $q^c_i\leq 10$,~ $i=0,\ldots, N$.  Such constraints are implemented by inserting inequality constraints (\ref{thm8}) to the MIP.  Figure \ref{figfinitequeue} shows the time-dependent queue sizes in the new optimal network flow profile. Notice that the queue associated with processor $a$ cannot be controlled since the inflow $f^a(\cdot)$ is fixed.

\begin{figure}[h!]
\begin{minipage}[b]{.49\textwidth}
\centering
\includegraphics[width=1\textwidth]{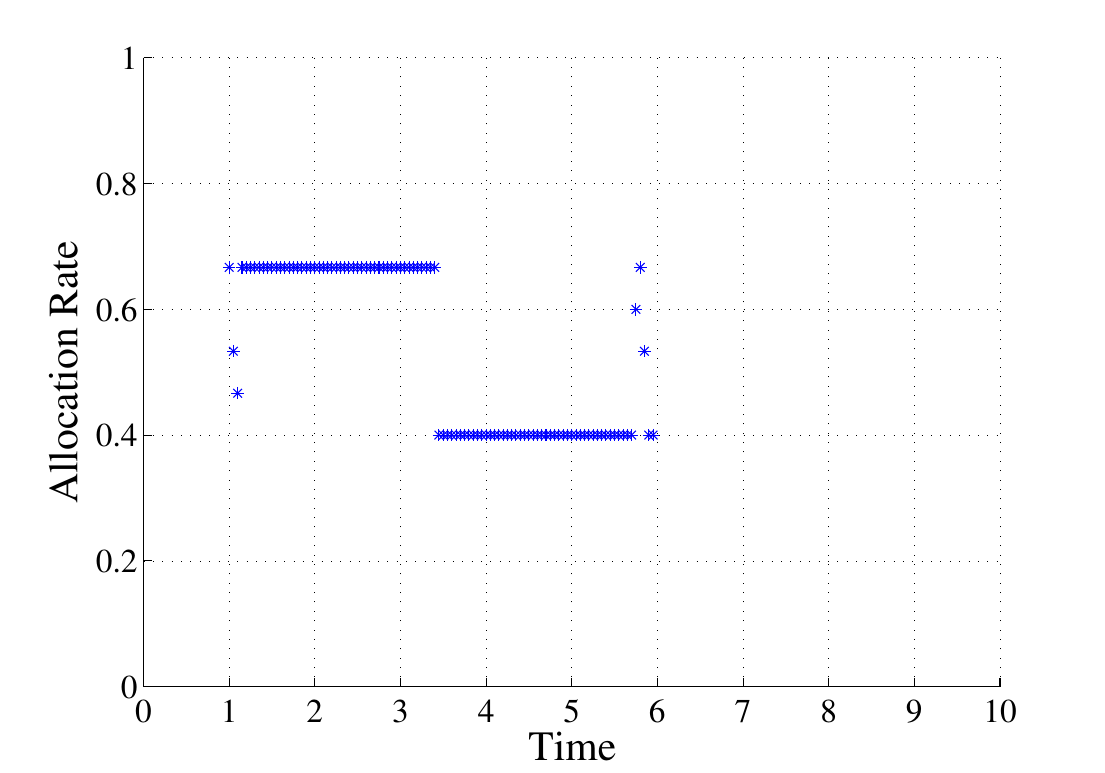}
\caption{\small With buffer size constraints: optimal allocation rate $A^{1,\,b}$}
\label{figfiniteA1}
\end{minipage}
\hspace{0.001cm}
\begin{minipage}[b]{.49\textwidth}
\centering
\includegraphics[width=1\textwidth]{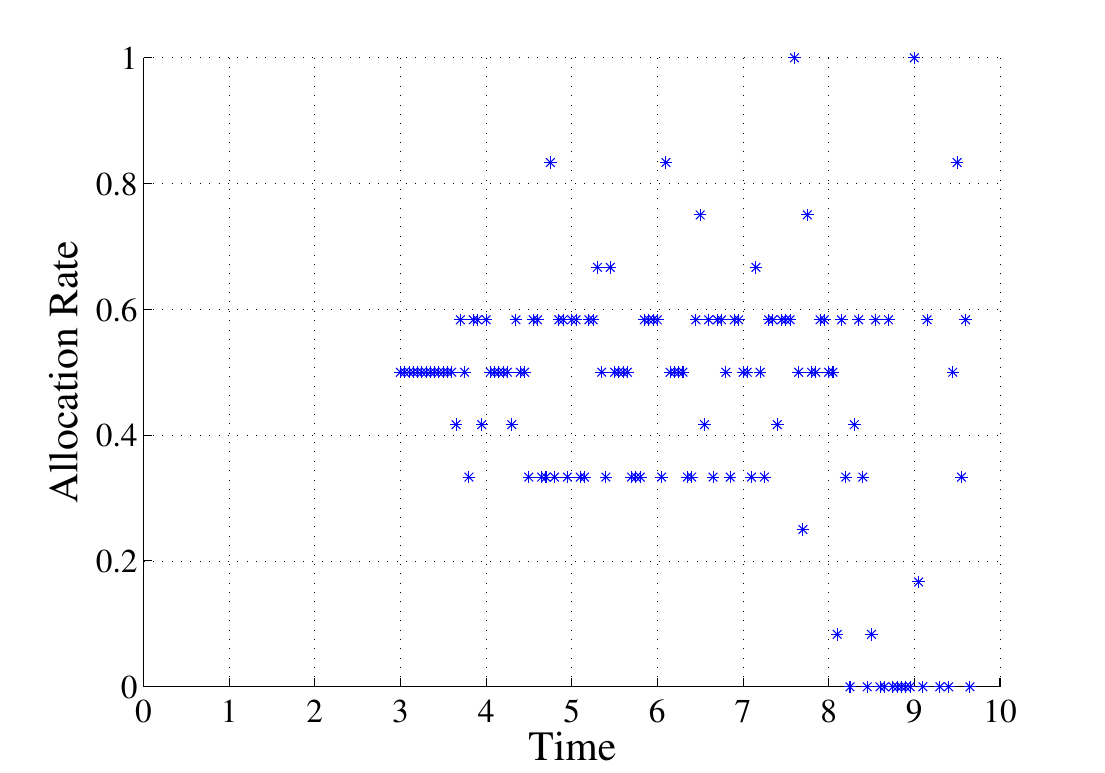}
\caption{\small With buffer size constraints: optimal allocation rate $A^{2,\,e}$}
\label{figfiniteA2}
\end{minipage}
\end{figure}
\begin{figure}[h!]
\centering
\includegraphics[width=0.75\textwidth]{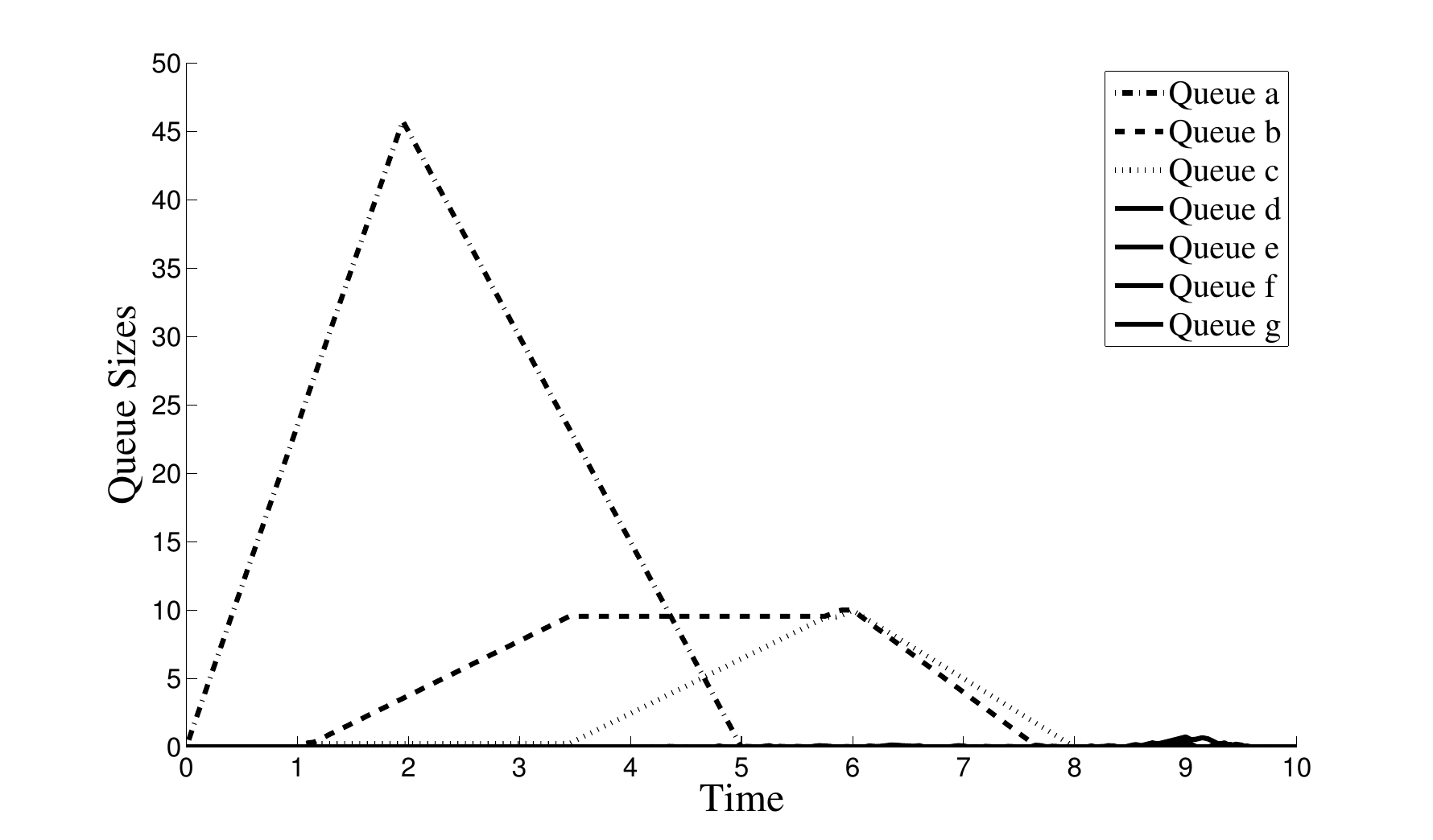}
\caption{With buffer size constraints:  queue upstream of each processor.}
\label{figfinitequeue}
\end{figure}

In the presence of the finite-buffer constraints,  the adjusted allocation rates are depicted in Figure \ref{figfiniteA1} and Figure \ref{figfiniteA2}.  Interestingly, the optimal value of $W^g(10)$ remains the same as the previous case, without buffer size constraints. This shows the non-uniqueness of the optimal solution of the MIP.

\subsection{Minimizing queuing}
In this example, we wish to minimize the queuing in the network, by controlling not only the allocation rates, but also the inflow profile $f^a(t)$ of the network. In view of our discussion in Section \ref{subsecinventory}, we consider the following revised objective function:
\begin{equation}\label{examplecost}
\max_{Q_i^a,\,A^{v,\,e}_i}\Big\{W^g(10)-\sum_{e\in\mathcal{A}}c_q^e\sum_{i=0}^N(Q_i^e-R_{i+\Delta^e}^e)\Big\}
\end{equation}
where the unit queuing cost $c_q^e$ is set to be equal to one for all $e\in\mathcal{A}$. By employing the above objective function, we wish to maximize the throughput of the network while keeping the queuing at a minimum level.  The resulting optimal inflow from the mixed integer program is shown in Figure \ref{figoptinflow}; the corresponding optimal allocation rates at vertices $1$ and $2$ are illustrated in Figure \ref{figA1fitting} and \ref{figA2fitting} respectively. Under such controls on the inflow and the allocation rates, the queues associated with all the processors remain empty throughout the time horizon.

The final throughput $W^g(10)$ is again 58.75.

\begin{figure}[h!]
\centering
\includegraphics[width=0.7\textwidth]{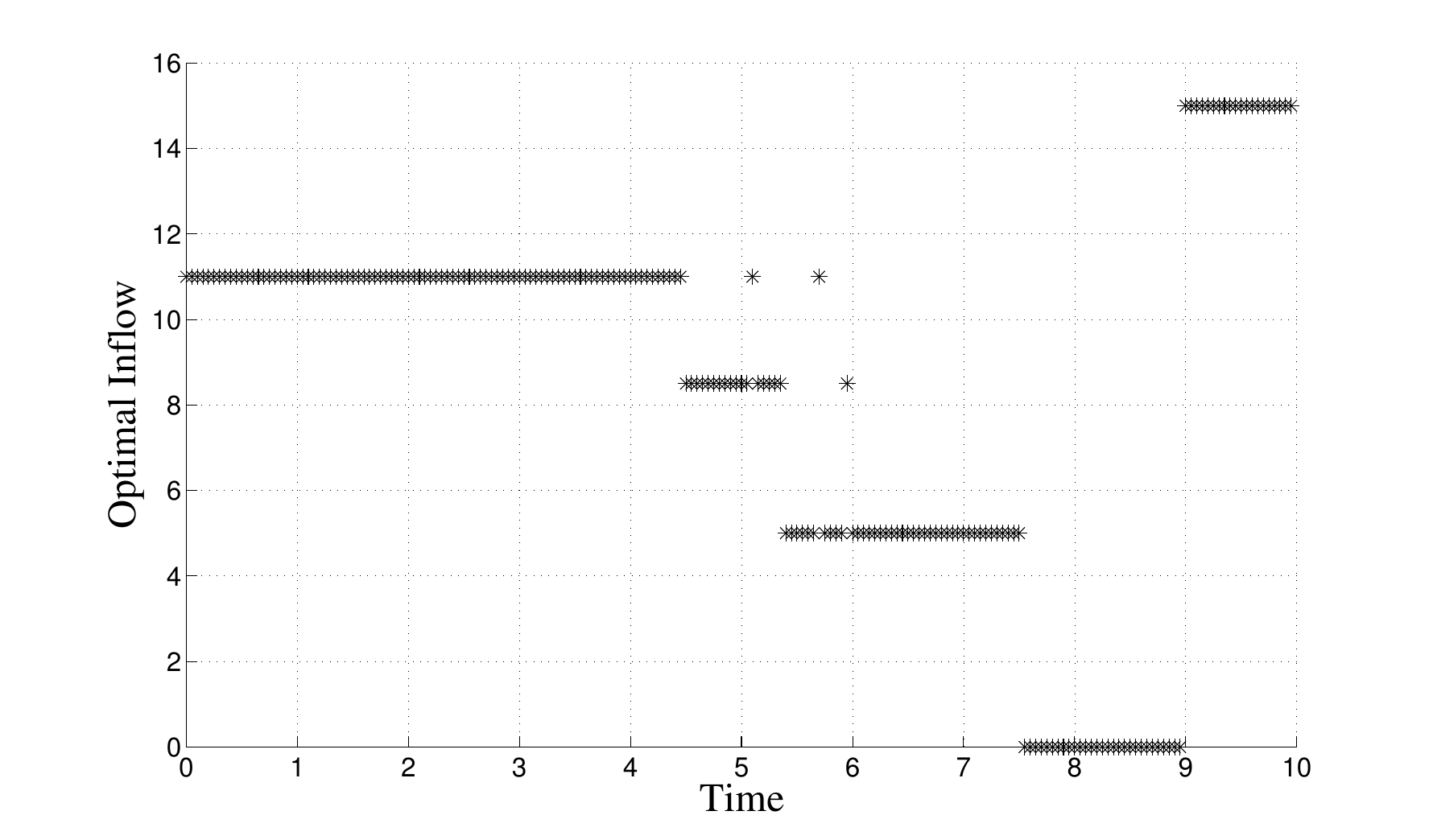}
\caption{Minimum queuing case: optimal inflow $f^a(t)$.}
\label{figoptinflow}
\end{figure}

\begin{figure}[h!]
\begin{minipage}[b]{.49\textwidth}
\centering
\includegraphics[width=1\textwidth]{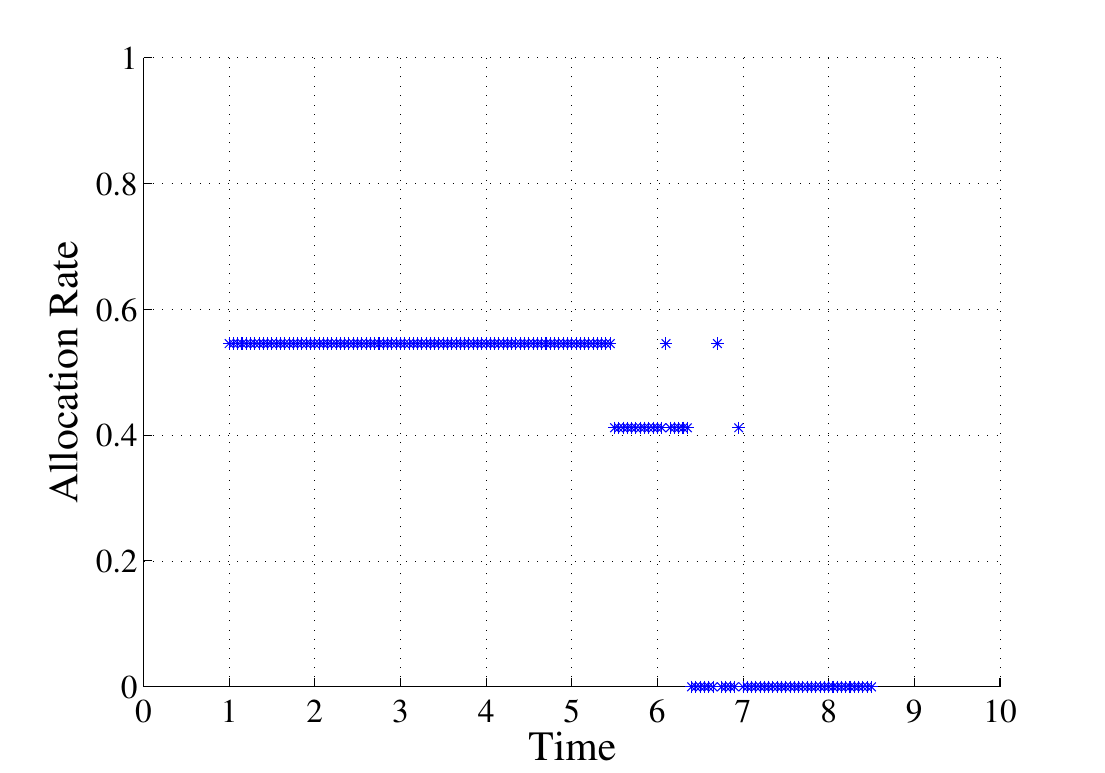}
\caption{\small Minimum queuing case:  allocation rate $A^{1,\,b}$}
\label{figA1fitting}
\end{minipage}
\hspace{0.001cm}
\begin{minipage}[b]{.49\textwidth}
\centering
\includegraphics[width=1\textwidth]{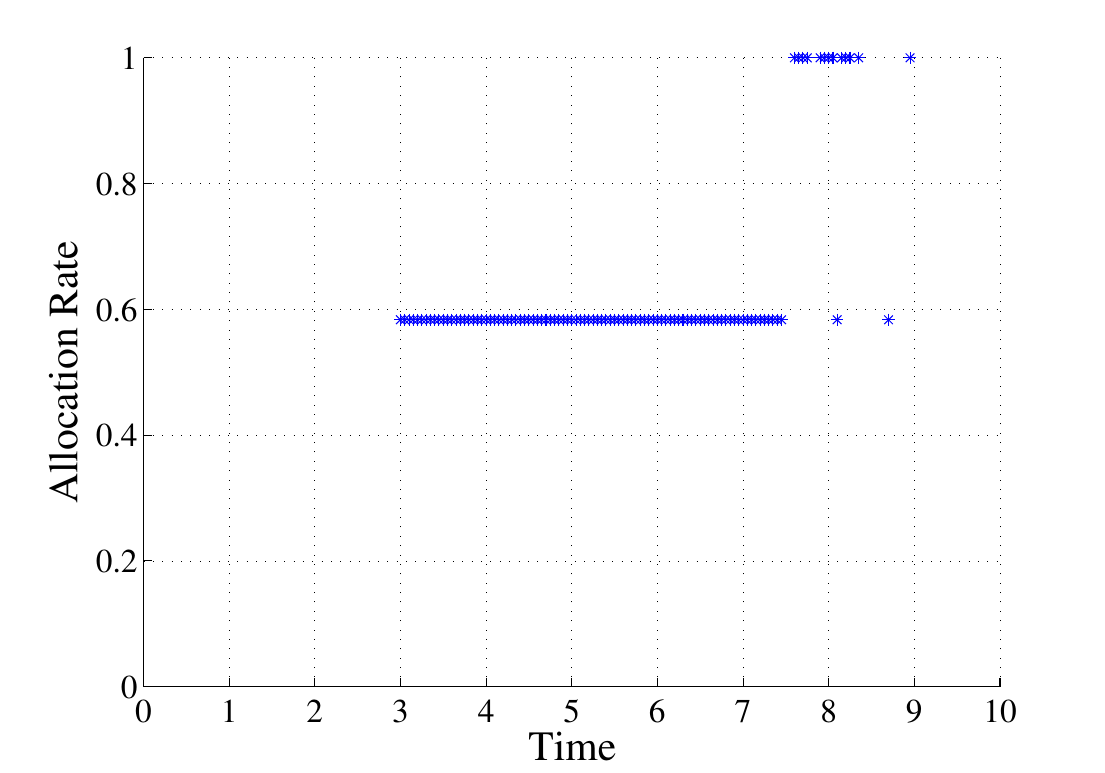}
\caption{\small Minimum queuing case: allocation rate $A^{2,\,e}$}
\label{figA2fitting}
\end{minipage}
\end{figure}


\subsection{Solution quality}\label{CSCsolutionprecision}

 We will analyze the result of our proposed MIP  in terms of numerical error and convergence, which will be compare with the finite difference schemes \cite{CSC, FGHKM, GHK1, GHK2}.   We use the same network and parameters as before, and assume infinite buffer capacities. The time horizon of interest is set to be $[0,\,80]$. The objective is to maximize the throughput  $W^g_N$. The inflow profile is
\begin{equation}\label{inflowQ}
 f^a(t)~=~\begin{cases} 45\qquad &0~\leq~t~<~10\\
0  \qquad & 10~\leq~ t~\leq~80\end{cases}
\end{equation}
Note that the integration of $f^a(t)$ on the time horizon is 450; it is natural to expect that the total throughput $W^g(80)$  is the same provided the time horizon is large enough. We run instances of the mixed integer program (\ref{thm1})-(\ref{thm6}) for $N=40,\,50,\,60,\,\ldots,\, 1000$ where $N$ is the number of time intervals. The optimal throughputs produced by these programs are plotted in Figure \ref{zigzag}.

\begin{figure}[h!]
\centering
\includegraphics[width=0.8\textwidth]{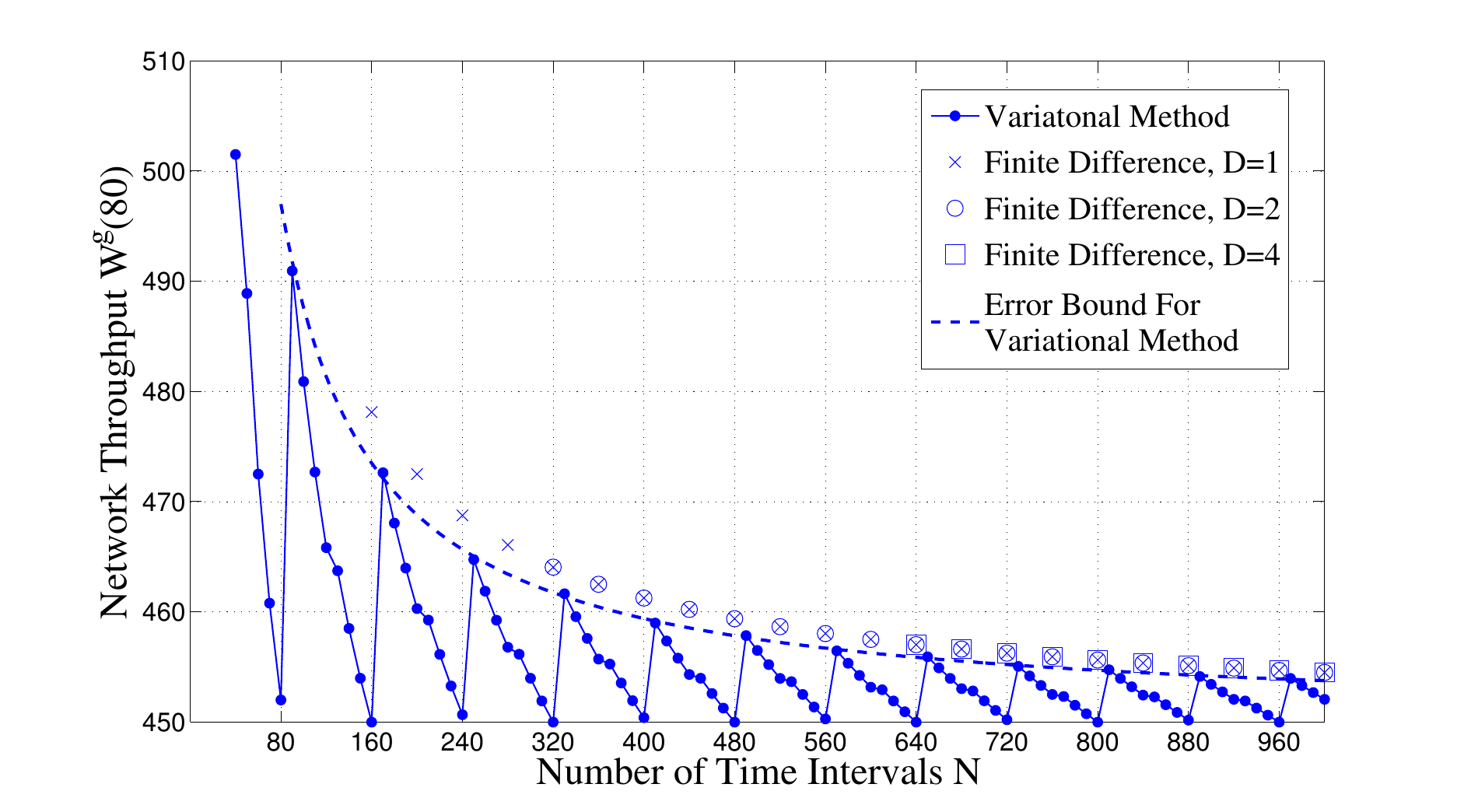}
\caption{Optimal output for different choices of time grid}
\label{zigzag}
\end{figure}

Before we explain the figure, let us recall from Proposition \ref{discrete1}  that the discrete-time Lax formula is exact provided that the throughput time $L^e/V^e,\,\,e\in\mathcal{A}$ is a multiple of the time step $h$, that is, when $N=160,\,320,\,480,\,640,\,800,\,960$ in our example. Otherwise, the error is given by estimate (\ref{error}):
\begin{equation}\label{error'}
0~\leq~W^e_{N}- W^e(80)~\leq~\Big(\lceil{L^e\over V^eh}\rceil-{L^e\over V^eh}\Big)\,h\mu^e\qquad e\in\mathcal{A}
\end{equation}
\eqref{error'} suggests that 1) the discrete-time numerical values overestimate the actual value of $W^g(80)$ which is 450; 2) the error displays an oscillatory pattern with damping, as $N$ increases. The oscillation is caused by the discontinuous function $\lceil{L\over Vh}\rceil$ while the damping is due to the factor $h$. Notice that the right hand side of error estimate (\ref{error'}) can be further relaxed to $h\mu$. It is not difficult to extend the error estimates for a single processor to a network, noticing that the error in the cumulative production curves adds up linearly through one processor (arc) to another:
\begin{equation}\label{heuristic}
0~\leq~W^g_N-W^g(80)~\leq~h\max_{p\in\mathcal{P}}\sum_{e\in p}\mu^e
\end{equation}
where $p=\{e_1,\,e_2,\ldots,\,e_m\}$ is any viable path of product flow, $e_i\in\mathcal{A}$, and $W^g(80)$ is the exact value of throughput. The right hand side of \eqref{heuristic} is shown in Figure \ref{zigzag} against different values of $N$.

 For comparison purposes, we also implement the MIP proposed by \cite{FGHKM}, which is based on a finite difference scheme. We experiment with different number of spatial intervals  $D$ for the discretization; for example, $D=1$ means the two-point discretization. The optimal throughputs are plotted in Figure \ref{zigzag} as well. Notice that when $D$ increases, the value of $N$ must also increase due to the CFL condition.  The numerical results summarized in Figure \ref{zigzag} leads to two observations: 1) the values of $D\in\{1,\,2,\,4\}$ seem to have no effect on the error as long as the CFL condition is satisfied; 2) in terms of solution precision, the variational approach outperforms the finite difference scheme even in the worst case scenario.

The above numerical experiment suggests that in order to maintain the same level of numerical error,  the size of our proposed MIP can be significantly smaller than the finite-difference approach. For example, to keep the error of $W^g_N$ under $1\%$, that is, $W^g_N\in[450,\,454.5]$, the finite-difference approach requires a time grid of approximately 1000 points and the same amount of binary variables, while with the variational approach, a time grid of 80 or 160 points is sufficient to achieve such solution precision. This observation is clearly supported by Figure \ref{zigzag}.

\subsection{A case study of a smoothed out version of \eqref{eqn2}-\eqref{feed}}\label{CSCcasestudy}
\cite{FGHKM} prosed a smoothed out version of the ODE (\ref{eqn2})-(\ref{feed}), in order to  avoid the numerical difficulties caused by the discontinuous right hand side. The modified ODE reads as follows.
\begin{equation}\label{smooth1}
{d\over dt}q^e(t)~=~\overline u^e(t)-f^e\big(\rho^e(t,\,a^e)\big)
\end{equation}
\begin{equation}\label{smooth2}
f^e\big(\rho^e(t,\,a^e)\big)~=~\min\Big\{\mu^e,\, {q^e(t)\over \varepsilon}\Big\}
\end{equation}
where $\varepsilon>0$ is a smoothing parameter. \eqref{smooth1}-\eqref{smooth2} is a practical representation of the original dynamic since it makes the right hand side of the ODE continuous, and the solution approximates the one of \eqref{eqn2}-\eqref{feed} well in most cases. To ensure numerical stability of the finite difference discretization, a stiffness condition (\cite{FGHKM}) requires that 
\begin{equation}\label{stiffcond}
\Delta t~\leq~ \varepsilon
\end{equation}
where $\Delta t$ denotes the time step size. 

In this case study, we compare the solutions of \eqref{smooth1}-\eqref{smooth2} and the variational formulation to illustrate that the smoothed out version of the original ODE could still yield ill-behaved solution in certain case.

Let us return to the example in Section \ref{nofinitebuffer} and set the network inflow to be
\begin{equation}\label{inflowQ'}
f^a(t)~=~\begin{cases} 14\qquad &0~\leq~t~<~10\\
0\qquad & 10~\leq~ t~\leq~80\end{cases}
\end{equation}
The MIP proposed by \cite{FGHKM} employs the smoothed \eqref{smooth1}-\eqref{smooth2} for the queue dynamics. Such MIP  is solved with $N=600,\,D=1$, where $N$ denotes the number of time steps, and $D$ denotes the number of spatial intervals for each processor. In other words, the two-point upwind discretization for the conservation law is used. The inflow, exit flow  and queue size of processor $a$ are shown together on the left part of Figure \ref{figcasestudy}. We notice that a discretization of (\ref{smooth1})-(\ref{smooth2}) yields a nonzero queue even though the inflow is strictly below the processor capacity $\mu^a=15$. It turns out that such non-physical queue is caused by the smoothing parameter $\varepsilon$ and can be reduced by choosing smaller $\varepsilon$, however, this once again implies a trade-off between numerical accuracy and computational burden, due to the stiffness condition \eqref{stiffcond}.

The  variational method, on the other hand, handles the same problem well with $N=600$.  The right part of Figure \ref{figcasestudy} shows that the queue stays zero and the exit flow is a a simple time shift of the inflow profile, which is consistent with the physics of the model.
\begin{figure}[htbp]
\begin{minipage}[b]{.49\textwidth}
\centering
\includegraphics[width=1\textwidth]{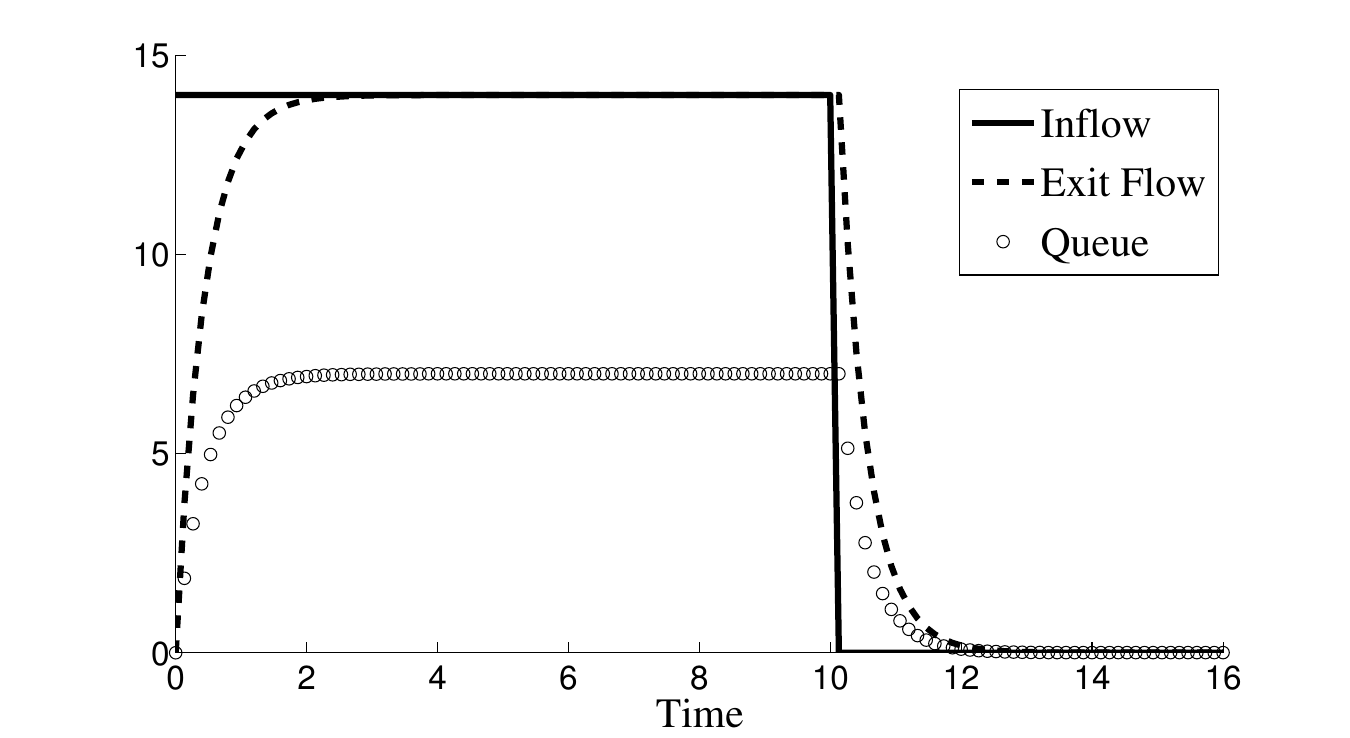}
\end{minipage}
\hspace{0.001cm}
\begin{minipage}[b]{.49\textwidth}
\centering
\includegraphics[width=1\textwidth]{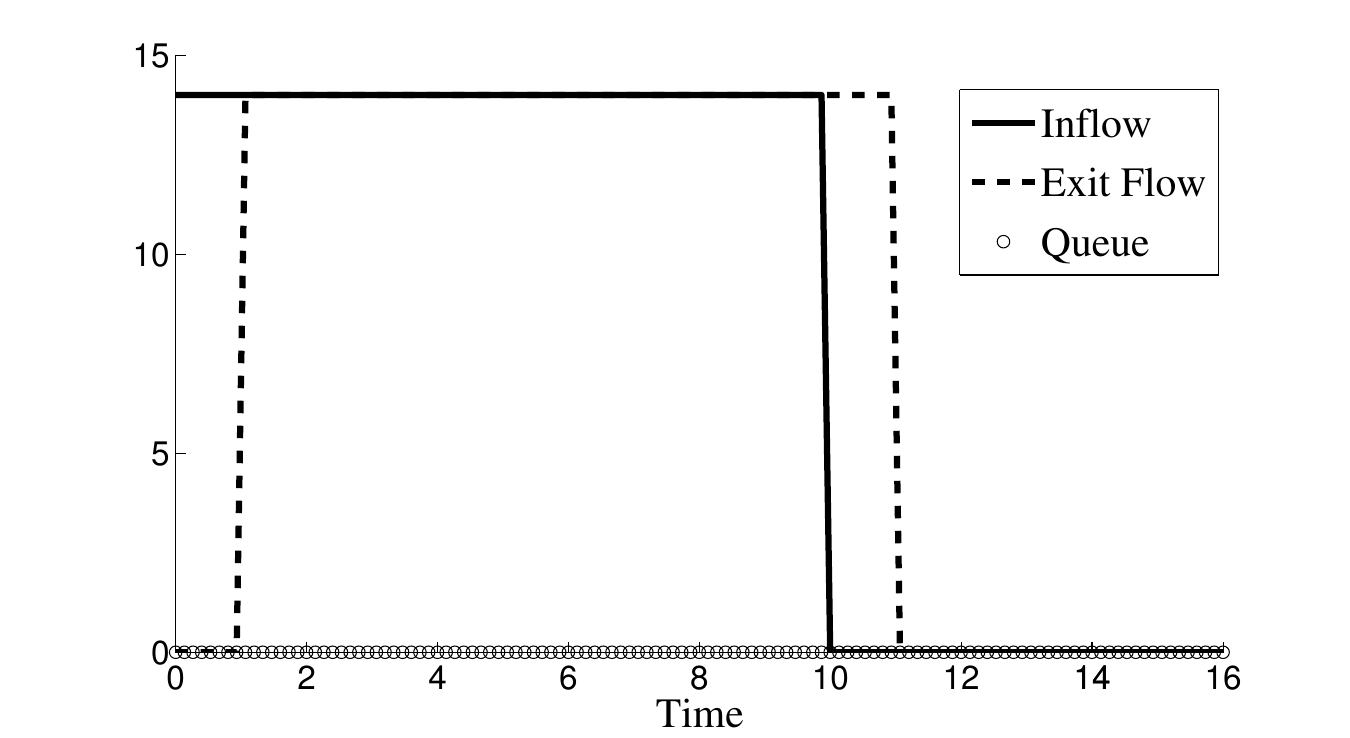}
\end{minipage}
\label{figcasestudy}
\caption{A comparison of  the smoothed out version of \eqref{eqn2}-\eqref{feed} and the variational approach. Left: solution on processor $a$ obtained by the finite-difference discretization of  \eqref{smooth1}-\eqref{smooth2} with $\varepsilon=0.5$. Right: solution on processor $a$ given by the variational approach.}
\end{figure}

\subsection{Solution time}\label{sectime}
As our final test, the computational times of both MIPs are recorded for the same network optimization problem as in Section \ref{nofinitebuffer} with the following network inflow: 
$$
f^a(t)~=~\begin{cases} 30 \qquad  0~\leq~t~\leq~5\\
0\,~\qquad 5~<~t~\leq~10\end{cases}
$$
For the MIP of \cite{FGHKM}, we employ a coarse two-point spatial discretization. For both MIP formulations, the same objective function is chosen to be 
\begin{equation}\label{obj2}
\max\sum_{i=0}^N {w^g_i\over 1+t_i}
\end{equation}
where $w^g_i$ is the exit flow on processor $g$ at time $t_i$.  Choosing such objective function ensures that the network throughput is maximized at every instance of time; in other words, the products are handled in a way such that they exit the network as early as possible.

The  two MIPs are solved with $N$, the number of time intervals, ranging from $160$ to $3000$. The corresponding solution times are summarized in Figure \ref{figcomruntime}.  Although the two MIPs are similar in size as we demonstrated in Section \ref{CSCcomparison}, the solution times of the proposed MIP problem is significantly lower than the other MIP. In addition, we observe a nearly linear growth of the computational time when $N$ increases for the variational approach.  Possible mechanisms for causing such significant difference in solution times for large scale problems,  involving either internal structure of the discretization or specific settings of the branch-and-bound algorithm, is currently under investigation.

\begin{figure}[h!]
\centering
\includegraphics[width=0.8\textwidth]{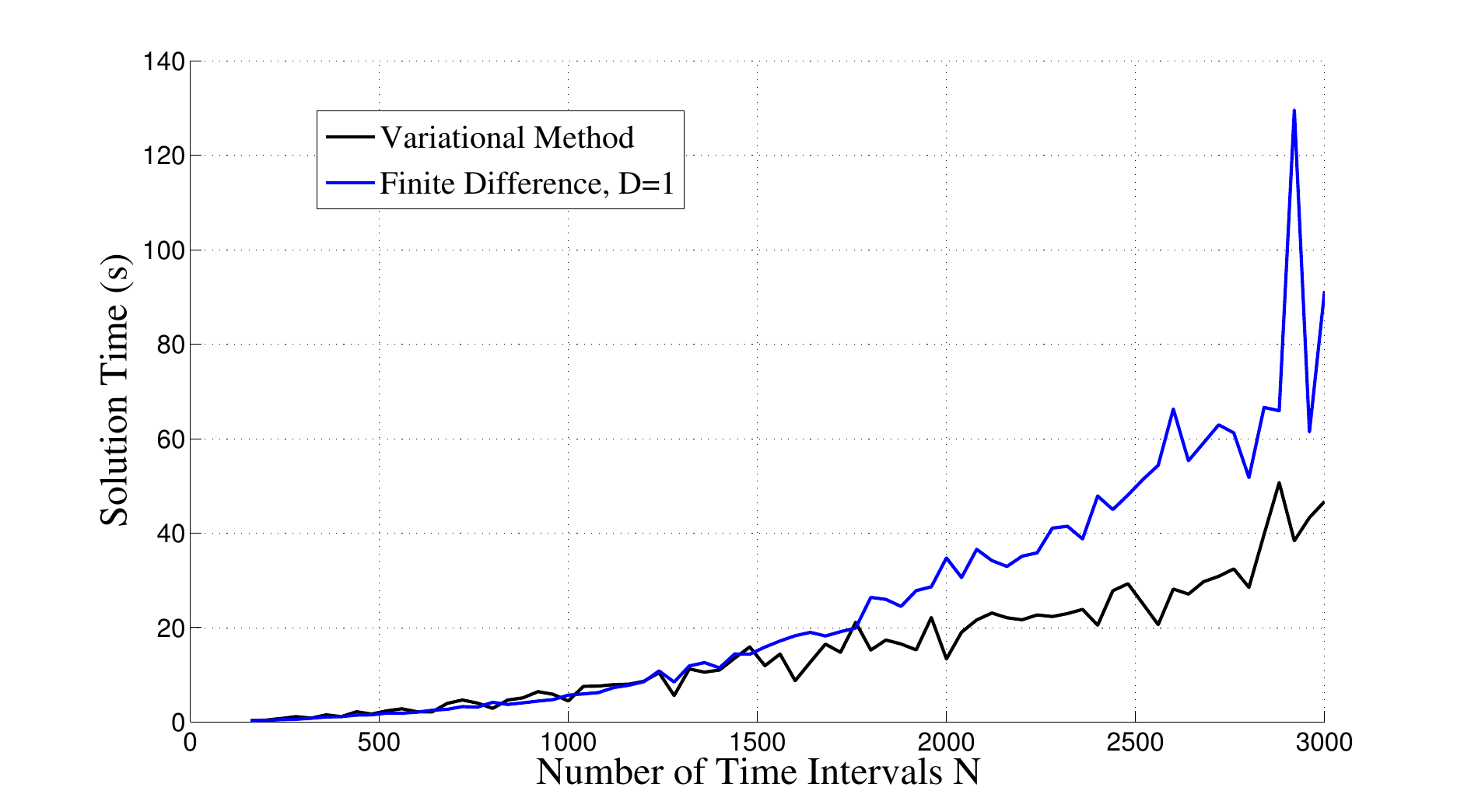}
\caption{Computational times of the two MIPs.}
\label{figcomruntime}
\end{figure}

\section{Conclusion}
This paper proposes a variational method for the modeling, computation, and optimization of a class of continuous supply chain networks. Such supply chain networks are investigated by \cite{GHK1} and can be formulated as a system of partial differential equations and ordinary differential equations. The main contribution made in this paper is an explicit solution representation of the dynamics on both the buffer queue and the processor. Our methodological framework is based on a Hamilton-Jacobi equation and a variational method known as the Lax formula. The closed-form solution is derived in both continuous and discrete time; and the latter leads to an algorithm with provable error estimates.  The proposed computational method is grid-free in the sense that it does not require a spatial discretization. Notably, the algorithm requires less computational effort and induces less numerical error than the finite difference method proposed for the coupling PDE and ODE \cite{FGHKM}. We also propose, based on the variational formulation, a mixed integer programming approach for the optimization of continuous supply chain networks. As we demonstrate in a series of numerical studies, the appropriate choice of the time grid could lead to a significantly reduced and even zero error, when compared with the MIP of \cite{FGHKM}. We also show that the proposed MIP requires much less computational effort than the one based on the PDE-ODE system, in order to properly represent the dynamics.

It is worthy mentioning that the continuous supply chain model, expressed by an ODE for the buffer queue and a PDE for the processor, is in many ways similar to the famous Vickrey  model \cite{Vickrey} for dynamic traffic flows. Applications of the variational approach in the venue of traffic modeling are presented in \cite{GVM1} and \cite{GVM2}.

Future extensions of the variational formulation will be focused on the modeling of multi-commodity supply chain network with more realistic features. To do so, it is desirable to consider inhomogeneous Hamilton-Jacobi equations to account for non-constant processing time. More sophisticated junction models need to be introduced as well to treat product flows with given origins and destinations.

\end{document}